\documentclass[12pt]{article}%
\usepackage{color}
\usepackage{amssymb, graphicx, amsmath}
\usepackage{amsmath,amsthm,amssymb}
\usepackage{booktabs}
\usepackage{cases}
\usepackage{subfigure}
\usepackage{graphicx}
\usepackage{CJK,graphicx}
\usepackage[figuresright]{rotating}

\usepackage{multirow}
\usepackage{listings}
\usepackage{float}
\usepackage{adjustbox}
\usepackage{pdflscape}
\usepackage{epstopdf}
\usepackage{rotating}
\usepackage{url}
\usepackage{hyperref}

\definecolor{myblue}{rgb}{0,0,0.5}
\definecolor{mygreen}{rgb}{0,0.5,0}
\definecolor{myred}{rgb}{0.5,0,0}
\newcommand{\D}{\mathbb{D}iv \;}

\usepackage{ulem}

\newtheorem{theorem}{Theorem}[section]

\newtheorem{lemma}{Lemma}[section]

\newtheorem{remark}{Remark}

\makeatletter
\def\hlinew#1{%
  \noalign{\ifnum0=`}\fi\hrule \@height #1 \futurelet
   \reserved@a\@xhline}

\def \[{\begin{equation}}
\def \]{\end{equation}}
\begin{document}
\begin{CJK*}{GBK}{song}

\begin{center}

{\large \bf  Indefinite linearized  augmented Lagrangian method for convex programming with linear inequality constraints}\\

\bigskip
\medskip

 {\bf Bingsheng He}\footnote{\parbox[t]{16cm}{
 Department of Mathematics,  Nanjing University, China.
  This author was supported by the NSFC Grant 11871029. Email: hebma@nju.edu.cn}}
 \quad
 {\bf Shengjie Xu}\footnote{\parbox[t]{16cm}{
 Department of Mathematics, Harbin Institute of Technology, Harbin,  China, and Department of Mathematics,   Southern  University of Science and Technology, Shenzhen, China. This author was supported by the NSFC grant 11871264 and the Guangdong Basic and Applied Basic Research Foundation of Chinathrough grant 2018A0303130123.   Email: xsjnsu@163.com
  }}
  \quad
 {\bf Jing Yuan}\footnote{\parbox[t]{16cm}{
  School of Mathematics and Statistics, Xidian University, Xi'an, China. Email: jyuan@xidian.edu.cn
  }}

\bigskip

\today

\end{center}

{\small

\parbox{0.95\hsize}{

\hrule

\medskip

{\bf Abstract.} The augmented Lagrangian method (ALM) is a benchmark for convex programming problems with linear constraints; ALM and its variants for linearly equality-constrained convex minimization models have been well studied in the literature.  However, much less attention has been paid to ALM for efficiently solving linearly inequality-constrained convex minimization models.  In this paper, we exploit an enlightening reformulation  of the newly developed indefinite linearized ALM for the equality-constrained convex optimization problem,  and present a new indefinite linearized ALM scheme for efficiently solving the convex optimization problem with linear inequality constraints. The proposed method enjoys great advantages, especially for large-scale optimization cases, in two folds mainly: first, it largely simplifies the challenging key subproblem of the classic ALM by employing its linearized reformulation, while keeping low complexity in computation; second, we show that only a smaller proximity regularization term is needed for provable convergence, which allows a bigger step-size and hence significantly better performance. Moreover, we show the global convergence of the proposed scheme upon its equivalent compact expression of prediction-correction, along with a worst-case $\mathcal{O}(1/N)$ convergence rate. Numerical results on some application problems demonstrate that a smaller regularization term can lead to a better experimental performance, which further confirms the theoretical results presented in this study.

\medskip

 \medskip

  \hrule

  }}

\bigskip

%%%%%%%%%%%%%%%%%%%%%%%%%%%%%%%%%%%%%%%%%%%%%%%%%%%%%%%%%%%%%%%%%%%%%%%%%%%%%%%%%%%%%%%%%%%%%%%%%%%%%%%%%%%%%%%
%%%%%%%%%%%%%%%%%%%%%%%%%%%%%%%%%%%%%%%%%%%%%%%%%%%%%%%%%%%%%%%%%%%%%%%%%%%%%%%%%%%%%%%%%%%%%%%%%%%%%%%%%%%%%%%

\section{Introduction}

A fundamental optimization model is the canonical convex programming problem with linear inequality constraints:
\begin{equation}\label{problem}
\min \big\{ \theta(x) \; | \; Ax\geq b, \; x\in{\cal X}\big\},
\end{equation}
where $\theta$: $\Re^n\to \Re$ is a closed proper convex but not necessarily smooth function,  ${\cal X} \subset \Re^n$ is a nonempty closed convex set, $A\in\Re^{m\times n}$ and $b\in\Re^m$.  Throughout our discussion,  the solution set of \eqref{problem} is assumed to be nonempty, and  $\rho(\cdot)$ is used  to stand for the spectrum of a matrix.

The model \eqref{problem} finds many applications in, e.g., linear and nonlinear programming problems~\cite{bazaraa2006nonlinear,boyd2004convex,luenberger1973introduction},  variational image processing models~\cite{yuan2010study,yuan2010continuous,yuan2014spatially,yuan2018modern} and some machine learning problems~\cite{cortes1995support,cristianini2000introduction,SNW2011}. In particular, the classic linearly equality-constrained convex optimization problem
\begin{equation}\label{problem1}
  \min\big\{ \theta(x) \; | \; Ax= b, \; x\in{\cal X} \big\},
\end{equation}
can be simply taken as its special case, for which the augmented Lagrangian method (ALM) proposed in \cite{Hes1969,Pow1969} was developed as a fundamental tool by imposing the quadratic penalty term $\beta\|Ax-b\|^2/2$ on the linear equality constraints for constructing the augmented Lagrangian function w.r.t. \eqref{problem1}:
\begin{equation}\label{aug-lagr}
\mathcal{L}_{\beta}^{\mathrm{E}}(x,\lambda)= \theta(x)-\lambda^T(Ax-b)+\frac{\beta}{2}\|Ax-b\|^2,
\end{equation}
where $\beta>0$ and $\lambda\in\Re^m$  is the Lagrange multiplier; moreover, with given $\lambda^k$,  the iterative scheme of ALM for \eqref{problem1} is
\begin{subequations}\label{ALM}
\begin{numcases}{\hbox{(Equality ALM)\;}}%\hbox{(Equality-constrained ALM)}
      \label{ALM-x} {x}^{k+1}  = \arg\min \bigl\{{\cal L}_{\beta}^{\mathrm{E}}(x,\lambda^k)  \; | \;  x\in\mathcal{X} \bigr\},  \\[0.10cm]
      \label{ALM-y} {{\lambda}}^{k+1} = \lambda^k -\beta(Ax^{k+1}-b) .
     \end{numcases}
\end{subequations}
The method \eqref{ALM} plays a significant role in both theoretical and algorithmic aspects for a large number of convex minimization problems.  We refer the readers to, e.g., \cite{bertsekas1996constrained,Bert2015,birgin2014practical,fortin1983augmented,glowinski1989augmented,ito2008lagrange}, for some monographs about the ALM and its variants. In particular, it was shown in  \cite{rockafellar1976augmented,rockafellar1976} that the ALM \eqref{ALM} is essentially an application of the proximal point algorithm \cite{martinet1970breve} to the dual of \eqref{problem1}.

Likewise, the classic way to solve the studied model \eqref{problem} is to apply the inequality version of  ALM  (see, e.g., \cite{bertsekas1996constrained,Bert2015}), where its iterative scheme is
\begin{subequations}\label{IIALM}
\begin{numcases}{\hbox{(Inequality ALM)\;}}
  \label{IIALM-x}x^{k+1} = \arg\min\big\{ \mathcal{L}_\beta^{\,\mathrm{I\!\!I}}(x,\lambda^k)  \; | \;  x\in {\cal X} \big\}, \\[0.10cm]
\label{IIALM-y}\lambda^{k+1} = [\lambda^k-\beta(Ax^{k+1}-b)]_{+},
  \end{numcases}
\end{subequations}
with the defined augmented Lagrangian function for \eqref{problem}:
\begin{equation}\label{I-AL}
  \mathcal{L}_\beta^{\,\mathrm{I\!\!I}}(x,\lambda):=\theta(x)+\frac{1}{2\beta}\Big\{\|[\lambda-\beta(Ax-b)]_{+}\|^2-\lambda^T\lambda\Big\}.
\end{equation}
Here,  $[\cdot]_{+}$ denotes the standard projection operator on the non-negative orthant in Euclidean space.   Clearly, the essential step for implementing the inequality-constrained  ALM  \eqref{IIALM} is to tackle the subproblem \eqref{IIALM-x}. However, discern that the gradient of the smoothness term over $x$
$$\nabla_x(\|[\lambda^k-\beta(Ax-b)]_{+}\|^2)=-2\beta A^T[\lambda^k-\beta(Ax-b)]_{+}$$
is non-smooth. This is the main obstacle for efficiently implementing the inequality ALM \eqref{IIALM}, and the primary purpose of this paper is to focus on overcoming this difficulty.

Recent development on the linearized  equality ALM (see Section \ref{sec:IDL-ALM}), which adds an additional proximal term w.r.t. an underlying matrix $L$ (see \eqref{L}) to linearize the  augmented Lagrangian function \eqref{aug-lagr} so as to properly reduce the computational difficulty of the pivotal subproblem \eqref{ALM-x}, shows that the positive definiteness of the added proximal term for linearization is not necessarily needed for rendering convergence. It thus results in an efficient indefinite linearized equality ALM for solving \eqref{problem1}, which allows a bigger step-size to accelerate the convergence. Motivated by this discovery, alternatively, the primary purpose of this paper is to develop an indefinite linearized inequality ALM for efficiently solving the convex minimization problem \eqref{problem} with linear inequality constraints.

To the best of our knowledge, this is the first work to study the indefinite linearized ALM framework for efficiently tackling the linearly inequality-constrained convex minimization problem \eqref{problem}. The proposed method enjoys great advantages in many folds: first,  compared with the classic inequality ALM \eqref{IIALM}, the new algorithm provides a much easier and more effective optimization strategy for tackling its core $x$-subproblem, so as to keep the same computational complexity as the linearized equality ALM;
second, the introduced relaxation of the proximal regularization term allows a bigger step-size, and hence better performance; third, the new algorithm covers the most recent indefinite linearized equality ALM as its special case.  Also, we establish the convergence analysis for the proposed method based on its equivalent prediction-correction interpretation, and further numerically illustrate its efficiency by extended experiments on the support vector machine for classification and continuous max-flow models for image segmentation.

The rest of this paper is organized as follows. In Section \ref{section2}, we summarize some preliminaries to motivate the proposed method for the studied model.  In Section \ref{section3}, we propose an indefinite linearized inequality ALM for convex programming problems with linear inequality constraints, which is rooted in a new reformulation of the most recent indefinite linearized equality ALM. We exploit a smaller regularization term for the proposed method, and establish its convergence theory in Section \ref{section4}.  The numerical experiments are further conducted in Section \ref{section5}, which validates our theoretical results numerically. Finally, some extensions and conclusions are discussed in Section \ref{section6} and Section \ref{section7}, respectively.

\section{Preliminaries}\label{section2}
\setcounter{equation}{0}

In this section,  we summarize some preliminary results for further analysis.

\subsection{Variational inequality characterization}
Similar to our previous work such as \cite{he2016convergence,he20121}, the analysis of this work will be conducted in the variational inequality (VI) context. Let us first recall a basic lemma regarding the VI optimality condition, whose proof is elementary and can be found in, e.g., \cite{beck2017first}.
\begin{lemma}  \label{blemma}
Let $\theta:\Re^n\rightarrow\Re$ and $f:\Re^n\rightarrow\Re$ be convex functions, and let ${\cal X} \subset \Re^n$ be a closed convex set.  If  $f$ is differentiable on an open set which contains ${\cal X}$ and the solution set of the optimization problem
$\min\{\theta(x)+f(x) \; | \; x\in \mathcal{X} \}$ is nonempty,  then we have
$$   x^\ast  \in  \arg\min \big\{ \theta(x) +f(x)  \;\big | \;  x\in {\cal X}\big\}$$
if and only if
$${x^\ast} \in {\cal X}, \;\;  \theta(x) -\theta(x^\ast) + (x-x^\ast)^T\nabla f(x^\ast) \ge 0,  \;\;  \forall \; x\in {\cal X}.$$
\end{lemma}

Then we derive the equivalent optimality condition of \eqref{problem} in the VI context. To this end, by adding the Lagrange multiplier $\lambda\in\Re_{+}^m$  to the inequality constraints,  the Lagrange function of  \eqref{problem} is
\begin{equation}\label{lagran}
  L(x,\lambda)=\theta(x)-\lambda^T(Ax-b).
\end{equation}
A point $(x^\ast,\lambda^\ast)\in\mathcal{X}\times \Re_{+}^m$  is called a saddle point of \eqref{lagran} if it satisfies
$$L_{\lambda\in\Re_{+}^m}(x^{\ast},\lambda)\leq L(x^{\ast},\lambda^{\ast})\leq L_{x\in \mathcal{X}}(x,\lambda^{\ast}).$$
According to Lemma \ref{blemma},  these two inequalities can be alternatively rewritten as
\begin{equation}\label{VI}
  \left\{ \begin{array}{lrl}
     x^\ast\in {\cal X}, & \theta(x) - \theta(x^\ast) + (x-x^\ast)^T(- { A}^T\lambda^\ast) \ge 0, & \forall\; x\in {\cal X};\\[0.2cm]
   \lambda^\ast\in \Re_{+}^m,  & (\lambda-\lambda^\ast)^T({ A}x^\ast-b)\ge 0,  &  \forall \; \lambda\in \Re_{+}^m,
        \end{array} \right.
\end{equation}
or more compactly,
\begin{equation}\label{VIU}
  \hbox{VI}(\Omega,F,\theta):\quad w^{\ast}\in\Omega,  \quad \theta(x)-\theta(x^\ast)+(w-w^\ast)^TF(w^\ast)\geq0, \quad \forall\; w\in\Omega,
\end{equation}
by setting
\begin{equation}\label{CVI}
  \Omega= {\cal X} \times \Re_{+}^m, \quad
  w = \left(\!\!\begin{array}{c}
                     x\\
                  \lambda \end{array}\!\! \right)
  \quad  \hbox{and} \quad
    F(w) =\left(\!\!\begin{array}{c}
     - { A}^T\lambda \\
     {A}x-b \end{array}\!\! \right).
\end{equation}
Since the operator $F(\cdot)$ in \eqref{CVI} is affine with a skew-symmetric matrix, it implies that
\begin{equation}\label{FMON}
  (u-v)^T(F(u)-F(v))\equiv0,\quad \forall \;u,v\in\Omega,
\end{equation}
which means that $F$ is monotone.  Throughout, we denote by $\Omega^\ast$ the solution set of \eqref{VIU}, which is also the saddle point set of the Lagrange function \eqref{lagran} for the studied model \eqref{problem}.

\subsection{Indefinite linearized ALM for the linearly equality-constrained convex programming}\label{sec:IDL-ALM}

It is clear that the essential step for implementing the equality ALM  \eqref{ALM} is to solve the subproblem \eqref{ALM-x}, for which the proximal ALM (see, e.g., \cite{he2020optimal,Zhang2010}) provides a significant strategy by considering the proximal regularized version of ALM:
\begin{subequations}\label{ALM1}
\begin{numcases}{\hbox{(Proximal ALM)}}
\label{ALM1-x} x^{k+1} =\arg\min
  \Big\{ \mathcal{L}_{\beta}^{\mathrm{E}}(x,\lambda^k)+\frac{1}{2}\|x-x^k\|_D^2 \; \big|\;  x\in {\cal X}   \Big\}, \\[0.1cm]
\label{ALM1-y}\lambda^{k+1} = \lambda^k-\beta(Ax^{k+1}-b),
  \end{numcases}
\end{subequations}
with $\|x\|_D^2:=x^TDx$ for any symmetric matrix $D$. As discussed in \cite{he2020optimal}, the proximal matrix $D\in\Re^{n\times n}$ in \eqref{ALM1} is required to be symmetric and positive-definite (i.e., $\|\cdot\|_D$ is a norm) for provable convergence. To implement the subproblem \eqref{ALM1-x}, by ignoring some constant terms, we have
\begin{eqnarray}\label{Prox-A-S}
 x^{k+1} &=&\arg\min_{x\in {\cal X}} \Big\{ \theta(x)-\lambda^k(Ax-b)+\frac{\beta}{2}\|Ax-b\|^2+\frac{1}{2}\|x-x^k\|_D^2\Big\}    \nonumber\\
  &=&  \arg\min_{x\in {\cal X}} \Big\{ \theta(x)-x^TA^T\lambda^k+\frac{\beta}{2}\|A(x-x^k)+Ax^k-b\|^2+\frac{1}{2}\|x-x^k\|_D^2\Big\}    \\
  &=&  \arg\min_{x\in {\cal X}} \Big\{ \theta(x)-x^TA^T[\lambda^k-\beta(Ax^k-b)]+\frac{\beta}{2}\|A(x-x^k)\|^2+\frac{1}{2}\|x-x^k\|_D^2\Big\}. \nonumber
\end{eqnarray}

Clearly, the solution set of \eqref{Prox-A-S} is determined by the objective function $\theta$, the matrix $A$, the domain $\mathcal{X}$ and the proximal matrix $D$. Its solution in general, as discussed in e.g., \cite{eckstein2017approximate,he2020optimal}, can only be approximated by certain iterative approach. On the other hand, if the positive definite matrix $D$ in \eqref{ALM1} is taken as
\begin{equation}\label{L}
  D:=rI_n-\beta A^TA \quad  \hbox{with} \quad  r>\beta\rho(A^TA),
\end{equation}
the subproblem \eqref{ALM1-x} or \eqref{Prox-A-S} can be essentially reduced to
\begin{eqnarray}\label{L-A-S}
x^{k+1}&=&\arg\min \Big\{  \theta(x)-x^TA^T[\lambda^k-\beta(Ax^k-b)]+\frac{r}{2}\|x-x^k\|^2  \;\big | \;  x\in {\cal X} \Big \},   \nonumber   \\
&=&\arg\min \Big\{ \theta(x)+\frac{r}{2}\|x-x^k-\frac{1}{r}A^T[\lambda^k-\beta(Ax^k-b)]\|^2  \;\big | \;  x\in {\cal X} \Big \},
\end{eqnarray}
i.e., the proximity operator of $\theta$ when $\mathcal{X}=\Re^n$, and it thus leads to the following so-called linearized ALM:
\begin{subequations}\label{ALM11}
\begin{numcases}{\hbox{(Linearized ALM)}}
\label{ALM11-x} x^{k+1} =\arg\min\Big\{ \mathcal{L}_{\beta}^{\mathrm{E}}(x,\lambda^k)+\frac{1}{2}\|x-x^k\|_D^2 \; \big|\;  x\in {\cal X}   \Big\}, \\[0.1cm]
\label{ALM11-y}\lambda^{k+1} = \lambda^k-\beta(Ax^{k+1}-b).
  \end{numcases}
\end{subequations}
This linearization is particularly significant for the case when the proximity operator of $\theta$, which is defined via
\begin{equation}\label{ProxOP}
\hbox{Prox}_{\theta,\beta}(y)=\arg\min\Big\{\theta(x)+\frac{1}{2\beta}\|x-y\|^2\;\big|\;x\in\Re^n\Big\},
\end{equation}
has a closed-form solution. For instance, the solution of \eqref{ProxOP} can be given by the soft thresholding function  (see, e.g., Example 6.8 in \cite{beck2017first})  when $\theta=\|\cdot\|_1$. We also see the review paper \cite{parikh2014proximal} for more concrete examples.

Meanwhile, it is easy to discern that a smaller $r$ in \eqref{L-A-S} would result in a bigger step-size at each iteration, which can accelerate the convergence moderately. It thus motivates the following linearized ALM with the optimal step size proposed in \cite{he2020optimal}:
\begin{subequations}\label{IALM}
\begin{numcases}{\hbox{(IDL-ALM)}}
 \label{ILAM-X} x^{k+1} =\arg\min
  \bigl\{ \mathcal{L}_{\beta}^{\mathrm{E}}(x,\lambda^k)+\frac{1}{2}\|x-x^k\|_{D_0}^2  \;\big|\;  x\in {\cal X}   \bigr\}, \label{IALM1-x}\\[0.1cm]
\label{ILAM-Y}\lambda^{k+1} = \lambda^k-\beta(Ax^{k+1}-b),\label{IALM1-y}
  \end{numcases}
\end{subequations}
in which
\begin{equation}\label{IALM1-cond}
D_0:=\tau rI-\beta A^TA \;\; \hbox{with} \;\; r>\beta\rho(A^TA)\;\; \hbox{and} \;\; \tau\in(0.75,1).
\end{equation}
Moreover, with the positive definite matrix $D$ defined in \eqref{L}, $D_0$ can be denoted by
\begin{equation}\label{L0}
D_0=\tau D-(1-\tau)\beta A^TA.
\end{equation}
On the one hand, compared with $D$ defined in \eqref{L}, it is easy to see that the new proximal matrix $D_0$ can be indefinite when  $\tau\in(0.75,1)$, which means that $\|\cdot\|_{D_0}$ is not necessarily a norm. On the other hand, the subproblem \eqref{ILAM-X} can be linearized in the form of \eqref{L-A-S}. With this regard, we call  \eqref{IALM} the \textbf{indefinite linearized ALM} (abbreviated as IDL-ALM for short).

\subsection{New insights to the IDL-ALM (\ref{IALM})} \label{sec:IALM}
Now we give an enlightening reformulation of the IDL-ALM \eqref{IALM}, which brings some new theoretical insights to the proposed numerical scheme for the studied model \eqref{problem}. More concretely,
setting $\tilde{\lambda}^k=\lambda^k-\beta(Ax^k-b)$, the subproblem \eqref{ILAM-X} equals to
\begin{eqnarray*}
x^{k+1} &=&\arg\min_{x\in {\cal X}} \Big\{ \theta(x)-x^TA^T[\lambda^k-\beta(Ax^k-b)]+\frac{\beta}{2}\|A(x-x^k)\|^2+\frac{1}{2}\|x-x^k\|_{D_0}^2\Big\}\\
&=&\arg\min_{x\in {\cal X}} \Big\{ \theta(x)-(\tilde{\lambda}^k)^TAx+\frac{\beta}{2}\|A(x-x^k)\|^2+\frac{1}{2}\|x-x^k\|_{(\tau rI-\beta A^TA)}^2\Big\}\\
&=&\arg\min_{x\in {\cal X}}\Big\{ \theta(x)-(\tilde{\lambda}^k)^TAx+\frac{\tau r}{2}\|x-x^k\|^2\Big\},
\end{eqnarray*}
and the $\lambda$-subproblem \eqref{ILAM-Y} can be further rewritten as
$$\lambda^{k+1}=\lambda^k-\beta(Ax^{k+1}-b)=\lambda^k-\beta(Ax^k-b)+\beta A(x^k-x^{k+1})=\tilde{\lambda}^k+\beta A(x^k-x^{k+1}).$$
Consequently, we can rewrite the newly developed IDL-ALM \eqref{IALM} as
\begin{subequations}\label{I-L-ALM}
\begin{numcases}{}
\label{I-L-ALM-u} \tilde{\lambda}^{k} =  \lambda^k - \beta(Ax^k-b),\\[0.15cm]
  x^{k+1} =\arg\min\bigl\{ \theta(x)-(\tilde{\lambda}^k)^TAx+\frac{\tau r}{2}\|x-x^k\|^2 \;\big|\;  x\in {\cal X}   \bigr\}, \\[0.15cm]
\lambda^{k+1} =  \tilde{\lambda}^k + \beta A(x^k-x^{k+1}).
  \end{numcases}
\end{subequations}
Indeed, this new representation \eqref{I-L-ALM} of the IDL-ALM \eqref{IALM} motivates the main result of this work immediately.

\section{Indefinite linearized ALM for the linearly inequality-constrained convex programming}\label{section3}
\setcounter{equation}{0}

The primary purpose of the paper is to efficiently solve the linearly inequality-constrained convex optimization problem \eqref{problem}.
The essential step of the classic ALM approach for \eqref{problem} is to tackle the subproblem \eqref{IIALM-x}, which is, however, often complicated in practice and has no efficient solution in general. With this respect, we now introduce a new indefinite linearized ALM approach for the studied model \eqref{problem}, which is motivated by the IDL-ALM reformulation \eqref{I-L-ALM} and can properly avoid solving such a difficult subproblem \eqref{IIALM-x}.

More specifically, with given $(x^k,\lambda^k)$, our new algorithm generates $(x^{k+1},\lambda^{k+1})$ via
\begin{subequations}\label{I-IDL-ALM}
\begin{numcases}{}
\label{I-IDL-ALM-u} \tilde{\lambda}^{k} =  [\lambda^k - \beta(Ax^k-b)]_{+},\\[0.15cm]
  x^{k+1} = \arg\min\bigl\{ \theta(x)-(\tilde{\lambda}^k)^TAx+\frac{\tau r}{2}\|x-x^k\|^2 \;\big|\;  x\in {\cal X}   \bigr\}, \\[0.15cm]
\lambda^{k+1} =  \tilde{\lambda}^k + \beta A(x^k-x^{k+1}),
  \end{numcases}
where
\begin{equation}
r>\beta\rho(A^TA) \quad \hbox{and} \quad \tau\in(0.75,1).
\end{equation}
\end{subequations}
Compared with the classic inequality ALM \eqref{IIALM}, the proposed method \eqref{I-IDL-ALM} only needs to tackle a much easier proximal estimation at each iteration, which may have a direct close-form solver for many favorable application problems.  Also, recall the equivalent representation \eqref{I-L-ALM} for the IDL-ALM \eqref{IALM}. It is clear that the minor difference between \eqref{I-IDL-ALM} and \eqref{I-L-ALM} is the corresponding $\tilde{\lambda}$ step, and that the original IDL-ALM  \eqref{IALM} can be regarded as a special case of \eqref{I-IDL-ALM} when the linearly equality-constrained model \eqref{problem1} is considered (i.e., the multiplier $\lambda$ is free).
With this regard, we also name \eqref{I-IDL-ALM} the \textbf{inequality version of indefinite linearized ALM} (abbreviated as I-IDL-ALM).

\subsection{Prediction-correction interpretation of the I-IDL-ALM (\ref{I-IDL-ALM})}\label{subsec2.3}
To simplify the analysis for the new algorithm \eqref{I-IDL-ALM}, we first rewrite \eqref{I-IDL-ALM} as a prediction-correction method as follows.
 \begin{equation}\label{I-L-ALM-P}
\hbox{(Prediction Step)} \left\{ \begin{array}{rcl}
 \tilde{\lambda}^{k}& = & [\lambda^k - \beta(Ax^k-b)]_{+},\\[0.2cm]
  \tilde{x}^{k} &=& \arg\min\bigl\{ \theta(x)-(\tilde{\lambda}^k)^TAx+\frac{\tau r}{2}\|x-x^k\|^2  \; \big| \;  x\in {\cal X}   \bigr\},
  \end{array} \right.
 \end{equation}
\begin{equation}\label{corrector}
\hbox{(Correction Step)}\quad \left(\!\!
  \begin{array}{c}
    x^{k+1} \\
    \lambda^{k+1} \\
  \end{array}\!\!
\right)=\left(\!\!
  \begin{array}{c}
    x^{k} \\
    \lambda^{k} \\
  \end{array}\!\!
\right)-\left(\!\!
          \begin{array}{cc}
            I_n & 0 \\
            -\beta A & I_m \\
          \end{array}\!\!
        \right)\left(\!\!
  \begin{array}{c}
    x^k-\tilde{x}^{k} \\
    \lambda^k-\tilde{\lambda}^{k} \\
  \end{array}\!\!
\right). \qquad \qquad \quad
\end{equation}
We would emphasize that such a two-stage explanation \eqref{I-L-ALM-P}-\eqref{corrector} only serves for further theoretical analysis and there is no need to use it when implementing the new method \eqref{I-IDL-ALM}.

%To see the variational inequality structure of \eqref{I-L-ALM-P},
Then we derive the associated VI of the prediction step \eqref{I-L-ALM-P}. To this end, it follows from Lemma \ref{blemma} that the output of the $x$-subproblem in \eqref{I-L-ALM-P} satisfies
\begin{equation}\label{I-LL-ALMxv}
\tilde{x}^k\in\mathcal{X}, \;\; \theta(x)-\theta(\tilde{x}^k)+(x-\tilde{x}^k)^T\{-A^T\tilde{\lambda}^{k}+\tau r(\tilde{x}^k-x^k)\}\geq0, \;\; \forall \; x\in\mathcal{X}.
\end{equation}
For the $\lambda$-subproblem in \eqref{I-L-ALM-P}, according to the basic property of the projection operator onto a convex set  (see, e.g., Appendix B of \cite{luenberger1973introduction}), we have
$$ \tilde{\lambda}^{k}\in\Re_{+}^m,\;\; (\lambda-\tilde{\lambda}^{k})^T\{\lambda^k-\beta(Ax^k-b)-\tilde{\lambda}^k\}\leq0, \;\; \forall \; \lambda\in\Re_{+}^m,$$
which can be further rewritten as
\begin{equation}\label{I-LL-ALMY}
  \tilde{\lambda}^{k}\in\Re_{+}^m,\;\; (\lambda-\tilde{\lambda}^{k})^T\big\{(A\tilde{x}^{k}-b)-A(\tilde{x}^{k}-x^k)+\frac{1}{\beta}(\tilde{\lambda}^{k}-\lambda^k)\big\}\geq0, \;\;
  \forall \; \lambda\in\Re_{+}^m.
\end{equation}
Adding \eqref{I-LL-ALMxv} and \eqref{I-LL-ALMY}, and using the notations in \eqref{CVI},  the VI of \eqref{I-L-ALM-P} can be compactly written as

\begin{center}
 \fbox{\begin{minipage}{16.5cm}
 \medskip
(Prediction step)
\begin{equation}\label{vi1}
  \tilde{w}^k\in\Omega, \;\; \theta(x)-\theta(\tilde{x}^k)+(w-\tilde{w}^k)^TF(\tilde{w}^k)\geq(w-\tilde{w}^k)^TQ(w^k-\tilde{w}^k), \;\; \forall \; w\in\Omega,
\end{equation}
where
\begin{equation}\label{Q}
  Q=\left(\!\!
    \begin{array}{cc}
      \tau rI_n & 0 \\
      -A & \frac{1}{\beta}I_m \\
    \end{array}\!\!
  \right).
\end{equation}
\vspace{-3pt} %\smallskip
\end{minipage}}
\end{center}
\medskip
Moreover, the corrector \eqref{corrector} can be recursively given by
\medskip
\begin{center}
 \fbox{\begin{minipage}{16.5cm}
 \medskip
(Correction step)
\begin{equation}\label{vi2}
  w^{k+1}=w^k-M(w^k-\tilde{w}^k),
\end{equation}
where
\begin{equation}\label{M}
  M=\left(\!\!
      \begin{array}{cc}
        I_n & 0 \\
        -\beta A & I_m \\
      \end{array}\!\!
    \right).
\end{equation}
\vspace{-3pt} %\smallskip
\end{minipage}}
\end{center}
\medskip

\subsection{Some basic matrices}
To simplify the notations for further analysis, let us also define some basic matrices. More concretely, with the matrices $Q$ and $M$ defined in \eqref{Q}  and \eqref{M}, respectively, we first define $H:=QM^{-1}$. Then it holds that
\begin{equation}\label{H}H=QM^{-1}=\left(\!\!
    \begin{array}{cc}
      \tau rI_n & 0 \\
      -A & \frac{1}{\beta}I_m \\
    \end{array}\!\!
  \right)\left(\!\!
      \begin{array}{cc}
        I_n & 0 \\
        \beta A & I_m \\
      \end{array}\!\!
    \right)=\left(\!\!
               \begin{array}{cc}
                 \tau rI_n & 0 \\
                 0 & \frac{1}{\beta}I_m \\
               \end{array}\!\!
             \right).
\end{equation}
It is clear that the matrix $H$ is symmetric and positive-definite for any $\tau>0$, $r>0$ and $\beta>0$. Since $H=QM^{-1}$, we have
$$M^THM=M^TQ=\left(\!\!
      \begin{array}{cc}
        I_n & -\beta A^T \\
        0 & I_m \\
      \end{array}\!\!
    \right)\left(\!\!
    \begin{array}{cc}
      \tau rI_n & 0 \\
      -A & \frac{1}{\beta}I_m \\
    \end{array}\!\!
  \right)=\left(\!\!
            \begin{array}{cc}
              \tau rI_n+\beta A^TA & -A^T \\
              -A & \frac{1}{\beta}I_m \\
            \end{array}\!\!
          \right).
  $$
Moreover, we define $G:=Q^T+Q-M^THM$, then we get
\begin{eqnarray}\label{G}
 G & = &(Q^T+Q)-M^THM=\left(\!\!
    \begin{array}{cc}
      2\tau rI_n & -A^T \\
      -A & \frac{2}{\beta}I_m \\
    \end{array}\!\!
  \right)-\left(\!\!
            \begin{array}{cc}
              \tau rI_n+\beta A^TA & -A^T \\
              -A & \frac{1}{\beta}I_m \\
            \end{array}\!\!
          \right) \nonumber \\[0.2cm]
          &=&\left(\!\!
            \begin{array}{cc}
              \tau rI_n-\beta A^TA & 0 \\
              0 & \frac{1}{\beta}I_m \\
            \end{array}\!\!
          \right)\overset{\eqref{IALM1-cond}}{=}\left(\!\!
            \begin{array}{cc}
              D_0  & 0\\
              0 & \frac{1}{\beta}I_m \\
            \end{array}\!\!
          \right).
\end{eqnarray}
Note that $D_0$ defined in \eqref{IALM1-cond} is a symmetric but not necessarily positive-definite matrix.  $\|\cdot\|_G$ (where $\|u\|_G^2:=u^TGu$) is not necessarily a norm.

\section{Convergence analysis}\label{section4}
\setcounter{equation}{0}

In this section, we explore the in-depth choice of the regularization parameter $\tau$ for the proximal subproblem of $x^{k+1}$ in the proposed I-IDL-ALM \eqref{I-IDL-ALM}, which is the main numerical load for \eqref{I-IDL-ALM}, so as to reach its convergence efficiently with less iterations. We also establish the convergence theory for the proposed method  \eqref{I-IDL-ALM}.
% The analysis technique is motivated by the works in, e.g., \cite{he2020optimal,he2020optimally,he20121}.

\subsection{In-depth choice on $\tau$}\label{sec:4.1}
It is obvious that a smaller $\tau$ for the proximal optimization of $x^{k+1}$ in the proposed method \eqref{I-IDL-ALM} allows a bigger step size for the proximity approximation of $x^{k+1}$, which potentially saves iterations to render a faster convergence. Indeed, we show that the value of $\tau$ can be even less than 1, actually within $(0.75, 1)$, to achieve convergence, which exactly means that a positive indefinite matrix $D_0$, denoted by \eqref{IALM1-cond}, is taken for proximity approximation. To establish the global convergence theory of the proposed method  \eqref{I-IDL-ALM},  let us first prove an essential inequality.
\begin{lemma}\label{KLemma}
Let $\{w^k\}$ and $\{\tilde{w}^k\}$ be the sequences generated by the prediction-correction formulation \eqref{I-L-ALM-P}-\eqref{corrector} of the proposed  I-IDL-ALM \eqref{I-IDL-ALM} for the studied model \eqref{problem}. Then we have
\begin{eqnarray}\label{l1}
\lefteqn{\theta(x)-\theta(\tilde{x}^k)+(w-\tilde{w}^k)^TF({w})} \nonumber \\
  &\quad \geq & \frac{1}{2}\big\{\|w-w^{k+1}\|_H^2-\|w-w^{k}\|_H^2\big\}+\frac{1}{2}\|w^k-\tilde{w}^k\|_G^2, \;\;  \forall  \; w\in\Omega,
\end{eqnarray}
where $G$ is the matrix defined as in \eqref{G}.
\end{lemma}
\begin{proof}
First, it follows from \eqref{vi2} and \eqref{H} that $H(w^k-w^{k+1}) = Q(w^k-\tilde{w}^k)$. Then we have
\begin{equation}\label{Th4-1}
 (w-\tilde{w}^k)^TQ(w^k-\tilde{w}^k)= (w-\tilde{w}^k)^TH(w^k-w^{k+1}),
\end{equation}
and thus we obtain
\begin{eqnarray}\label{b1}
 & & \theta(x)-\theta(\tilde{x}^k)+(w-\tilde{w}^k)^TF({w}) \overset{\eqref{FMON}}{=}  \theta(x)-\theta(\tilde{x}^k)+(w-\tilde{w}^k)^TF(\tilde{w}^k) \nonumber \\
 & & \quad \overset{\eqref{vi1}}{\geq}  (w-\tilde{w}^k)^TQ(w^k-\tilde{w}^k) \overset{\eqref{Th4-1}}{=}  (w-\tilde{w}^k)^TH(w^k-{w}^{k+1}).
\end{eqnarray}
Applying the identity
$$(a-b)^TH(c-d)=\frac{1}{2}\big\{\|a-d\|_H^2-\|a-c\|_H^2\big\}+\frac{1}{2}\big\{\|c-b\|_H^2-\|d-b\|_H^2\big\}$$
to the right hand of \eqref{b1} with $a=w,\;b=\tilde{w}^k,\;c=w^k\; \hbox{and} \;d=w^{k+1}$, we obtain
\begin{eqnarray}\label{b2}
\lefteqn{(w-\tilde{w}^k)^TH(w^k-{w}^{k+1})} \nonumber\\
 &\quad =& \frac{1}{2}\{\|w-w^{k+1}\|_H^2-\|w-w^k\|_H^2\} +\frac{1}{2}\{\|w^k-\tilde{w}^k\|_H^2-\|w^{k+1}-\tilde{w}^k\|_H^2\}.
\end{eqnarray}
For the second term of \eqref{b2}, we have
\begin{eqnarray} \label{b3}
 \lefteqn{\|w^k-\tilde{w}^k\|_H^2-\|w^{k+1}-\tilde{w}^k\|_H^2} \nonumber \\
   &\overset{(\ref{vi2})}{=}& \|w^k-\tilde{w}^k\|_H^2-\|(w^{k}-\tilde{w}^k)-M(w^{k}-\tilde{w}^k)\|_H^2  \nonumber \\
   &=& 2(w^{k}-\tilde{w}^k)^THM(w^{k}-\tilde{w}^k)-(w^{k}-\tilde{w}^k)^TM^THM(w^{k}-\tilde{w}^k) \nonumber \\
   &=& (w^{k}-\tilde{w}^k)^T(Q^T+Q-M^THM)(w^{k}-\tilde{w}^k)  \nonumber \\
   &\overset{(\ref{G})}{=}& \|w^{k}-\tilde{w}^k\|_G^2.
\end{eqnarray}
Substituting \eqref{b2} and \eqref{b3} into \eqref{b1}, the assertion of the lemma follows immediately.
\end{proof}

If the matrix $G$ defined in (\ref{G}) is positive definite (it holds naturally when $\tau\geq1$), then Lemma \ref{KLemma} can essentially imply the global convergence and a worst-case convergence rate measured by iteration complexity for the proposed method \eqref{I-IDL-ALM}. We refer the readers to, e.g., \cite{he2018,he2014strictly,he20121,he2018class}, for the analogous analytical techniques. However, we are interested in whether a smaller $\tau$ is feasible, which means that the matrix $G$ is not necessarily positive-definite. Hence, the assertion of Lemma \ref{KLemma} can not be utilized directly, which makes the convergence analysis for the new algorithm \eqref{I-IDL-ALM} more challenging.

Motivated the analysis techniques in, e.g., \cite{he2016convergence,he2020optimal,he2020optimally},  our main goal in the following is to show the term  $\|w^{k}-\tilde{w}^k\|_G^2$ in \eqref{l1} satisfies
\begin{equation}\label{mianidea}
  \|w^{k}-\tilde{w}^k\|_G^2\geq \phi(w^k,w^{k+1})-\phi(w^{k-1},w^k)+\varphi(w^k,w^{k+1}),
\end{equation}
where $\phi(\cdot,\cdot)$ and $\varphi(\cdot,\cdot)$ are both  non-negative functions, and $\varphi(\cdot,\cdot)$ is used to measure how much $w^{k+1}$ fails to be a solution point of \eqref{VIU}. Once the above inequality is established,  by substituting \eqref{mianidea} into  \eqref{l1}, we can immediately obtain
\begin{eqnarray}\label{l-key}
  \lefteqn{\theta(x)-\theta(\tilde{x}^k)+(w-\tilde{w}^k)^TF({w})} \nonumber \\
   &\quad \geq  & \frac{1}{2}\big\{\|w-w^{k+1}\|_H^2+\phi(w^k,w^{k+1})\big\}-\frac{1}{2}\big\{\|w-w^{k}\|_H^2+\phi(w^{k-1},w^k)\big\} \nonumber \\
   &         &  +\, \varphi(w^k,w^{k+1}), \quad   \forall  \; w\in\Omega.
\end{eqnarray}
As will be shown in Section \ref{section4.2}, the inequality (\ref{l-key}) is an essential property for establishing the convergence analysis of the proposed method \eqref{I-IDL-ALM}.

To show the desired inequality \eqref{mianidea}, we first write $\|w^{k}-\tilde{w}^k\|_G^2$ as the sum of several terms.
\begin{lemma}
Let $\{w^k\}$ and $\{\tilde{w}^k\}$ be the sequences generated by the prediction-correction formulation \eqref{I-L-ALM-P}-\eqref{corrector} of the  I-IDL-ALM \eqref{I-IDL-ALM} for the studied model \eqref{problem}. Then we have
\begin{eqnarray}\label{b7}
  \|w^k-\tilde{w}^k\|_G^2&=&\tau\|x^k-{x}^{k+1}\|_D^2+\tau\beta\|A(x^k-{x}^{k+1})\|^2+\frac{1}{\beta}\|\lambda^k-\lambda^{k+1}\|^2  \nonumber\\
 && +\,2(\lambda^k-\lambda^{k+1})^TA(x^k-{x}^{k+1}).
\end{eqnarray}
\end{lemma}
\begin{proof}
To begin with, it follows from \eqref{G} and  \eqref{L0}  that
\begin{eqnarray}\label{Lema-1-1}
\|w^k-\tilde{w}^k\|_G^2&\overset{\eqref{G}}{=}&\|x^k-\tilde{x}^k\|_{D_0}^2+\frac{1}{\beta}\|\lambda^k-\tilde{\lambda}^k\|^2 \nonumber\\
  &\overset{\eqref{L0}}{=}&\tau\|x^k-\tilde{x}^k\|_D^2-(1-\tau)\beta\|A(x^k-\tilde{x}^k)\|^2+\frac{1}{\beta}\|\lambda^k-\tilde{\lambda}^k\|^2.
\end{eqnarray}
According to \eqref{corrector}, we have
\begin{equation}\label{Lema-1}
  \tilde{x}^k=x^{k+1} \quad \hbox{and} \quad \lambda^k-\tilde{\lambda}^k=(\lambda^k-\lambda^{k+1})+\beta A(x^k-x^{k+1}).
\end{equation}
Substituting \eqref{Lema-1} into \eqref{Lema-1-1}, we get
\begin{eqnarray*}
  \|w^k-\tilde{w}^k\|_G^2&=& \tau\|x^k-x^{k+1}\|_D^2-(1-\tau)\beta\|A(x^k-x^{k+1})\|^2\\
   &&+\frac{1}{\beta}\|(\lambda^k-\lambda^{k+1})+\beta A(x^k-x^{k+1})\|^2 \\
   &=& \tau\|x^k-{x}^{k+1}\|_D^2+\tau\beta\|A(x^k-{x}^{k+1})\|^2+\frac{1}{\beta}\|\lambda^k-\lambda^{k+1}\|^2 \\
   &  & +2(\lambda^k-\lambda^{k+1})^TA(x^k-{x}^{k+1}),
\end{eqnarray*}
 and the proof is complete.
\end{proof}

Now we turn to deal with the crossing term $2(\lambda^k-\lambda^{k+1})^TA(x^k-{x}^{k+1})$ in \eqref{b7} and estimate a lower-bound in the quadratic forms. Two various lower-bounds for the crossing term $(\lambda^k-\lambda^{k+1})^TA(x^k-{x}^{k+1})$ are given in the following two lemmas, respectively.

%\begin{eqnarray}\label{b8}
%% \nonumber to remove numbering (before each equation)
%  \lefteqn{(\lambda^k-\lambda^{k+1})^TA(x^k-x^{k+1})} \nonumber \\
%   & \qquad \geq & \frac{1}{2}\tau\|x^k-{x}^{k+1}\|_D^2-\frac{3}{2}(1-\tau)\beta\|A(x^k-x^{k+1})\|^2 -\frac{1}{2}\tau\|x^{k-1}-{x}^{k}\|_D^2  \nonumber\\
%  &      &  -\frac{1}{2}(1-\tau)\beta\|A(x^{k-1}-x^{k})\|^2.
%\end{eqnarray}

\begin{lemma}
Let $\{w^k\}$ and $\{\tilde{w}^k\}$ be the sequences generated by the prediction-correction formulation \eqref{I-L-ALM-P}-\eqref{corrector} of the I-IDL-ALM \eqref{I-IDL-ALM}  for solving  \eqref{problem}. Then we have
\begin{eqnarray}\label{b8}
  (\lambda^k-\lambda^{k+1})^TA(x^k-x^{k+1}) & \geq & \Big\{\frac{1}{2}\tau\|x^k-{x}^{k+1}\|_D^2+\frac{1}{2}(1-\tau)\beta\|A(x^k-x^{k+1})\|^2\Big\} \nonumber \\
   &     & -\Big\{\frac{1}{2}\tau\|x^{k-1}-{x}^{k}\|_D^2+\frac{1}{2}(1-\tau)\beta\|A(x^{k-1}-x^{k})\|^2\Big\} \nonumber \\
   &     & -2(1-\tau)\beta\|A(x^k-x^{k+1})\|^2.
\end{eqnarray}
\end{lemma}
\begin{proof}
Our first goal is to rewrite the inequality \eqref{I-LL-ALMxv} as a form which does not contain $\tilde{x}^k$ and $\tilde{\lambda}^k$. To this end, it follows from \eqref{Lema-1} that  $$\tilde{x}^{k}=x^{k+1}\quad  \hbox{and} \quad \tilde{\lambda}^{k} = \lambda^{k+1}-\beta A(x^{k}-x^{k+1}).$$
Recall the matrix $D_0$ defined in \eqref{L0}. It holds that
\begin{eqnarray*}
-A^T\tilde{\lambda}^{k}+\tau r(\tilde{x}^{k}-x^k) &=& -A^T(\lambda^{k+1}-\beta A(x^{k}-x^{k+1}))+\tau r(x^{k+1}-x^k) \\
   &=&  -A^T{\lambda}^{k+1}+D_0({x}^{k+1}-x^k).
\end{eqnarray*}
Therefore, the inequality \eqref{I-LL-ALMxv} can be rewritten as
\begin{equation}\label{l2}
  \theta(x)-\theta(x^{k+1})+(x-x^{k+1})^T\{-A^T\lambda^{k+1}+D_0(x^{k+1}-x^k)\} \geq 0, \;\; \forall \; x\in\mathcal{X}.
\end{equation}
Note that \eqref{l2} also holds for $k:=k-1$. We thus obtain
\begin{equation}\label{l3}
  \theta(x)-\theta(x^k)+(x-x^k)^T\{-A^T\lambda^k+D_0(x^k-x^{k-1})\} \geq 0,   \;\;   \forall   \;   x\in\mathcal{X}.
\end{equation}
Setting  $x=x^k$  and $x=x^{k+1}$ in  \eqref{l2} and \eqref{l3}, respectively, and adding them, we get
\begin{eqnarray}\label{c3-A}
  \lefteqn{(\lambda^k-\lambda^{k+1})^TA(x^k-x^{k+1})} \nonumber \\
   &\qquad \geq & \|x^k-x^{k+1}\|_{D_0}^2+(x^k-x^{k+1})^TD_0(x^k-x^{k-1}).
\end{eqnarray}
 For the first part of the right side of  \eqref{c3-A}, we have
 \begin{eqnarray}\label{c3-B}
   \|x^k-x^{k+1}\|_{D_0}^2 &\overset{\eqref{L0}}{=} &\|x^k-{x}^{k+1}\|_{(\tau D\!-\!(1-\tau)\beta A^TA)}^2 \nonumber \\
    &=& \tau\|x^k-x^{k+1}\|_D^2-(1-\tau)\beta \|A(x^k-x^{k+1})\|^2.
 \end{eqnarray}
 For the second part  of the right side of \eqref{c3-A}, using Cauchy-Schwarz inequality,  we get
 \begin{eqnarray}\label{c3-C}
  \lefteqn{(x^k-x^{k+1})^TD_0(x^k-x^{k-1})} \nonumber\\
    &\qquad=& (x^k-x^{k+1})^T(\tau D - (1-\tau)\beta A^TA)(x^k\!-\!x^{k-1})  \nonumber \\
    &\qquad=& \tau(x^k-x^{k+1})^TD(x^k-x^{k-1})  -(1-\tau)\beta(A(x^k-x^{k+1}))^TA(x^k-x^{k-1})   \nonumber\\
    & \qquad\geq & -\frac{1}{2}\tau\big\{\|x^k-x^{k+1}\|_D^2+\|x^{k-1}-x^{k}\|_D^2\big\}\nonumber \\
    &   &-\frac{1}{2}(1-\tau)\beta\big\{\|A(x^k-x^{k+1})\|^2+\|A(x^{k-1}-x^k)\|^2\big\}.
 \end{eqnarray}
Substituting \eqref{c3-B} and \eqref{c3-C} into \eqref{c3-A}, the assertion \eqref{b8}  follows immediately.
\end{proof}

%Another lower-bound of the term $(\lambda^k-\lambda^{k+1})^TA(x^k-x^{k+1})$ can be given via the following lemma.
\begin{lemma}
Let $\{w^k\}$  be the sequence generated by the prediction-correction scheme \eqref{I-L-ALM-P}-\eqref{corrector} of the proposed I-IDL-ALM \eqref{I-IDL-ALM} for solving the studied model \eqref{problem}. Then, for any $\tau\in[\frac{3}{4},1]$, we have
\begin{eqnarray}\label{b9}
  \lefteqn{(\lambda^k-\lambda^{k+1})^TA(x^k-x^{k+1})} \nonumber \\
   &\qquad \geq & -(\tau -\frac{1}{2})\beta\|A(x^k-x^{k+1})\|^2-(\frac{5}{2}-2\tau)\frac{1}{\beta}\|\lambda^k-\lambda^{k+1}\|^2.
\end{eqnarray}
\end{lemma}
\begin{proof}
Since $(\tau-\frac{1}{2}) >0$, it follows from Cauchy-Schwarz inequality that
\begin{eqnarray*}%\label{b9-A}
  \lefteqn{(\lambda^k-\lambda^{k+1})^TA(x^k-x^{k+1})}  \\
   &\qquad \geq & -(\tau -\frac{1}{2})\beta\|A(x^k-x^{k+1})\|^2-   \frac{1}{4(\tau-\frac{1}{2})}\frac{1}{\beta}\|\lambda^k-\lambda^{k+1}\|^2.
\end{eqnarray*}
Thus, it suffices to prove that
  \begin{equation*}   %\label{b9-B}
         \frac{1}{4(\tau-\frac{1}{2})}  \le   \frac{5}{2}-2\tau,   \quad   \forall  \; \tau\in[\frac{3}{4},1].
         \end{equation*}
Let $h(\tau):=(\tau -\frac{1}{2})(\frac{5}{2}-2\tau)$. It is easy to verify that $h(\tau)$ is a concave function and it reaches its minimum on the interval $[\frac{3}{4},1]$ at the endpoint $\tau=\frac{3}{4}$ or $\tau=1$.  Due to
  $h(\frac{3}{4}) = h(1) = \frac{1}{4} $,  we have
$$       h(\tau)=  (\tau -\frac{1}{2})(\frac{5}{2}-2\tau)  \ge  \frac{1}{4} , \quad \forall \; \tau\in[\frac{3}{4},1]. $$
This completes the proof of the lemma.
\end{proof}
Adding \eqref{b8} and \eqref{b9}, we get
\begin{eqnarray}\label{c7}
\lefteqn{2(\lambda^k-\lambda^{k+1})^TA(x^k-{x}^{k+1})} \nonumber \\
 &\geq& \big\{ \frac{1}{2}\tau\|x^k-{x}^{k+1}\|_D^2+\frac{1}{2}(1-\tau)\beta\|A(x^k-x^{k+1})\|^2\big\} \nonumber   \\
 & & -  \big\{\frac{1}{2}\tau\|x^{k-1}-{x}^{k}\|_D^2 +\frac{1}{2}(1-\tau)\beta\|A(x^{k-1}-x^{k})\|^2\big\} \nonumber \\
 &  &+ (\tau-\frac{3}{2})\beta\|A(x^k-x^{k+1})\|^2 -(\frac{5}{2} -2\tau)\frac{1}{\beta}\|\lambda^k-\lambda^{k+1}\|^2,
\end{eqnarray}
Then we obtain the following theorem immediately.
\begin{theorem}\label{uptheorem}
Let $\{w^k\}$ and $\{\tilde{w}^k\}$ be the sequences generated by the prediction-correction formulation \eqref{I-L-ALM-P}-\eqref{corrector} of the I-IDL-ALM \eqref{I-IDL-ALM} for  \eqref{problem}. Then, for any $\tau\in[\frac{3}{4},1]$, we have
\begin{eqnarray}\label{b15}
  \lefteqn{\|w^k-\tilde{w}^k\|_G^2}  \nonumber\\
  & \geq  & \frac{1}{2}\big\{\tau\|x^k-x^{k+1}\|_D^2+(1-\tau)\beta\|A(x^k-x^{k+1})\|^2\big\} \nonumber  \\
  & &-\frac{1}{2}\big\{\tau\|x^{k-1}-x^k\|_D^2+(1-\tau)\beta\|A(x^{k-1}-x^k)\|^2\big\} \nonumber \\
  &  &+\tau \|x^k-x^{k+1}\|_D^2+2(\tau-\frac{3}{4})\big\{\beta\|A(x^k-x^{k+1})\|^2+\frac{1}{\beta}\|\lambda^k-\lambda^{k+1}\|^2\big\}.
\end{eqnarray}
\end{theorem}
\begin{proof}
Substituting \eqref{c7} into \eqref{b7},  we obtain
\begin{eqnarray*}
 \|w^k-\tilde{w}^k\|_G^2& \geq & \tau \|x^k-{x}^{k+1}\|_D^2+\tau\beta\|A(x^k-{x}^{k+1})\|^2+  \frac{1}{\beta}\|\lambda^k-\lambda^{k+1}\|^2  \\
   &  & +\big\{ \frac{1}{2}\tau\|x^k-{x}^{k+1}\|_D^2+\frac{1}{2}(1-\tau)\beta\|A(x^k-x^{k+1})\|^2\big\}   \\
   &  &  -  \big\{\frac{1}{2}\tau\|x^{k-1}-{x}^{k}\|_D^2 +\frac{1}{2}(1-\tau)\beta\|A(x^{k-1}-x^{k})\|^2\big\}  \\
   &  &   +  (\tau-\frac{3}{2})\beta\|A(x^k-x^{k+1})\|^2 -(\frac{5}{2} -2\tau)\frac{1}{\beta}\|\lambda^k-\lambda^{k+1}\|^2.
\end{eqnarray*}
It can be further summarized as
\begin{eqnarray*}
  \lefteqn{\|w^k-\tilde{w}^k\|_G^2}  \\
  &\quad \geq  & \frac{1}{2}\big\{\tau\|x^k-x^{k+1}\|_D^2+(1-\tau)\beta\|A(x^k-x^{k+1})\|^2\big\}   \\
  &         &-\frac{1}{2}\big\{\tau\|x^{k-1}-x^k\|_D^2+(1-\tau)\beta\|A(x^{k-1}-x^k)\|^2\big\}  \\
  &         &+\,\tau \|x^k-x^{k+1}\|_D^2+2(\tau-\frac{3}{4})\big\{\beta\|A(x^k-x^{k+1})\|^2+\frac{1}{\beta}\|\lambda^k-\lambda^{k+1}\|^2\big\},
\end{eqnarray*}
and the proof is complete.
\end{proof}

Obviously, the assertion \eqref{b15} corresponds to the desirable inequality \eqref{mianidea} by defining the non-negative functions $\phi(\cdot,\cdot)$ and $\varphi(\cdot, \cdot)$ as below:
$$\phi(w^k,w^{k+1}):=\frac{1}{2}\big\{\tau\|x^k-x^{k+1}\|_D^2+(1-\tau)\beta\|A(x^k-x^{k+1})\|^2\big\}$$
and
$$\varphi(w^k,w^{k+1}):=\tau \|x^k-x^{k+1}\|_D^2+2(\tau-\frac{3}{4})\big\{\beta\|A(x^k-x^{k+1})\|^2+\frac{1}{\beta}\|\lambda^k-\lambda^{k+1}\|^2\big\}.$$

\subsection{Global convergence}\label{section4.2}
To show the global convergence of the I-IDL-ALM \eqref{I-IDL-ALM} for the convex programming problem \eqref{problem} with linear inequality constraints,  we first prove an essential inequality which paves the way to the convergence proof of \eqref{I-IDL-ALM} by the following lemma.

\begin{lemma}\label{lemma-convergence}
Let $\{w^k\}$ and $\{\tilde{w}^k\}$ be the sequences generated by the prediction-correction scheme \eqref{I-L-ALM-P}-\eqref{corrector} of the I-IDL-ALM \eqref{I-IDL-ALM} for solving  \eqref{problem}. Then, for any  $\tau\in[\frac{3}{4},1]$ and  $w^\ast\in\Omega^\ast$,  we have
\begin{eqnarray}\label{b11}
  \lefteqn{\|w^{k+1}-w^{\ast}\|_H^2+\frac{1}{2}\big\{\tau\|x^k-x^{k+1}\|_D^2+(1-\tau)\beta\|A(x^k-x^{k+1})\|^2\big\}} \nonumber \\
&\quad \leq & \|w^k-w^\ast\|_H^2+\frac{1}{2}\big\{\tau\|x^{k-1}-x^k\|_D^2+(1-\tau)\beta\|A(x^{k-1}-x^{k})\|^2\big\}  \nonumber\\
&       &-\,\big\{\tau \|x^k-x^{k+1}\|_D^2+2(\tau-\frac{3}{4})(\beta\|A(x^k-x^{k+1})\|^2+\frac{1}{\beta}\|\lambda^k-\lambda^{k+1}\|^2)\big\}.
\end{eqnarray}
\end{lemma}
\begin{proof}Substituting \eqref{b15} into \eqref{l1},  we have
\begin{eqnarray}\label{b10}
  \lefteqn{\theta(x)-\theta(\tilde{x}^k)+(w-\tilde{w}^k)^TF(w)} \nonumber \\
 & \quad  \geq & \big\{\frac{1}{2}\|w-w^{k+1}\|_H^2+\frac{1}{4}(\tau\|x^k-{x}^{k+1}\|_D^2+(1-\tau)\beta\|A(x^k-x^{k+1})\|^2)\big\} \nonumber    \\
 &         & -\, \big\{\frac{1}{2}\|w-w^{k}\|_H^2+\frac{1}{4}(\tau\|x^{k-1}-x^{k}\|_D^2+(1-\tau)\beta\|A(x^{k-1}-x^{k})\|^2)\big\}  \nonumber   \\
 &         & +\, \frac{\tau}{2} \|x^k-{x}^{k+1}\|_D^2+(\tau-\frac{3}{4})\big\{\beta\|A(x^k-x^{k+1})\|^2+\frac{1}{\beta}\|\lambda^k-\lambda^{k+1}\|^2\big\}.
\end{eqnarray}
Setting  $w$ as arbitrary $w^\ast\in\Omega^\ast$ in \eqref{b10},  it further implies that
\begin{eqnarray*}
 \lefteqn{\|w^k-w^\ast\|_H^2+\frac{1}{2}\big\{\tau\|x^{k-1}-x^{k}\|_D^2+(1-\tau)\beta\|A(x^{k-1}-x^{k})\|^2\big\}}  \\
 \lefteqn{-\,\big\{\|w^{k+1}-w^{\ast}\|_H^2+\frac{1}{2}(\tau\|x^k-x^{k+1}\|_D^2+(1-\tau)\beta\|A(x^k-x^{k+1})\|^2)\big\}}  \\
 &\qquad \geq & \tau \|x^k-{x}^{k+1}\|_D^2+2(\tau-\frac{3}{4})\big\{\beta\|A(x^k-x^{k+1})\|^2+\frac{1}{\beta}\|\lambda^k-\lambda^{k+1}\|^2\big\} \\
 &       &+\,2\big\{\theta(\tilde{x}^k)-\theta(x^\ast)+(\tilde{w}^k-w^\ast)^TF(w^\ast)\big\}.
\end{eqnarray*}
Note that
$$\theta(\tilde{x}^k)-\theta(x^\ast)+(\tilde{w}^k-w^\ast)^TF(w^\ast)\geq0.$$
The assertion of this lemma follows immediately.
\end{proof}

According to Lemma \ref{lemma-convergence}, now we are ready to show the global convergence of the proposed I-IDL-ALM \eqref{I-IDL-ALM}, and it is summarized in the following theorem.
\begin{theorem}\label{theo-convergence}
Let $\{w^k\}$ and $\{\tilde{w}^k\}$ be the sequences generated by the prediction-correction formulation \eqref{I-L-ALM-P}-\eqref{corrector} of the I-IDL-ALM \eqref{I-IDL-ALM} for solving the studied model \eqref{problem}. Then, for any $\tau\in(0.75,1)$,  the sequence $\{w^k\}$ converges to some $w^\infty\in\Omega^\ast$.
\end{theorem}
\begin{proof}
Our first goal is to show that the generated sequences $\{w^k\}$ and $\{\tilde{w}^k\}$ are bounded. To this end, summing \eqref{b11} over $k=1,2,\ldots,\infty$, we have
\begin{eqnarray*}
\lefteqn{\sum_{k=1}^{\infty}\big\{\tau \|x^k-x^{k+1}\|_D^2+2(\tau-\frac{3}{4})(\beta\|A(x^k-x^{k+1})\|^2+\frac{1}{\beta}\|\lambda^k-\lambda^{k+1}\|^2)\big\}}    \\
&\quad  \leq & \big\{\|w^1-w^\ast\|_H^2+\frac{1}{2}\tau\|x^0-x^1\|_D^2+\frac{1}{2}(1-\tau)\beta\|A(x^0-x^1)\|^2\big\},
\end{eqnarray*}
which further implies that
$$\lim_{k\rightarrow\infty}\tau \|x^k-x^{k+1}\|_D^2+2(\tau-\frac{3}{4})\big\{\beta\|A(x^k-x^{k+1})\|^2+\frac{1}{\beta}\|\lambda^k-\lambda^{k+1}\|^2\big\}=0.$$
Consequently, we have $\lim_{k\rightarrow\infty}\|x^k-x^{k+1}\|_D^2=0$ and $\lim_{k\rightarrow\infty}\|\lambda^k-\lambda^{k+1}\|^2=0$, which further implies that
\begin{equation}\label{keytail}
  \lim_{k\rightarrow\infty}\|w^k-w^{k+1}\|=0.
\end{equation}
Moreover, since the matrix $M$ is non-singular, it follows from \eqref{vi2} and \eqref{keytail}  that
 \begin{equation}\label{keytail1}
  \lim_{k\rightarrow\infty}\|w^k-\tilde{w}^k\|=0.
\end{equation}
For any $w^\ast\in\Omega^\ast$ and the integer $k\geq1$, it follows from \eqref{b11} that
\begin{eqnarray}\label{Fmonoton}
  \lefteqn{\|w^{k+1}-w^\ast\|_H^2} \nonumber  \\
  &\quad \leq & \|w^k-w^\ast\|_H^2+\frac{1}{2}\big\{\tau\|x^{k-1}-x^k\|_D^2+(1-\tau)\beta\|A(x^{k-1}-x^k)\|^2\big\}  \nonumber \\
  & \quad \leq & \cdots\leq\|w^1-w^\ast\|_H^2+\frac{1}{2}\big\{\tau\|x^0-x^1\|_D^2+(1-\tau)\beta\|A(x^0-x^1)\|^2\big\},
\end{eqnarray}
the sequence  $\{w^k\}$ is thus bounded. Moreover,  it follows from \eqref{keytail1} that the sequence  $\{\tilde{w}^k\}$ is also bounded.

Then, we show that the sequence $\{w^k\}$ converges to some $w^{\infty}\in\Omega^\ast$. Let $w^\infty$ be a cluster point of $\{\tilde{w}^k\}$, and $\{\tilde{w}^{k_j}\}$  be the associated subsequence converging to $w^\infty$. Combining \eqref{vi1} and \eqref{keytail1}, it follows from the continuity of $\theta$ and $F$ that
$$ \theta(x) -\theta(x^\infty) +(w-w^{\infty})^T F(w^{\infty}) \geq 0, \;\; \forall \;  w\in\Omega,$$
and thus $w^\infty\in\Omega^\ast$. Furthermore, according to \eqref{Fmonoton}, we have
$$\|w^{k+1}-w^\infty\|_H^2\leq\|w^k-w^\infty\|_H^2+\frac{1}{2}\big\{\tau\|x^{k-1}-x^k\|_D^2+(1-\tau)\beta\|A(x^{k-1}-x^k)\|^2\big\}.$$
Hence, the sequence $\{\|w^k - w^{\infty}\|_H^2\}_{k \geq 0}$ is nonincreasing, and it is bounded away below from zero.  Moreover, it follows from $\lim_{j\rightarrow\infty}\tilde{w}^{k_j}=w^\infty$ and \eqref{keytail1} that $\lim_{j\rightarrow\infty}w^{k_j}=w^\infty$. Therefore, we have $\{w^k\}$ converges to $w^\infty$ and the proof is complete.
\end{proof}

\begin{remark}
Following the same analysis techniques in, e.g., \cite{he2020optimal,he20121}, it is trivial to prove that the proposed method  \eqref{I-IDL-ALM} also enjoys a worst-case $\mathcal{O}(1/N)$ convergence rate measured by the iteration complexity. Here, we opt to omit these meticulous details for succinctness.
\end{remark}

\section{Numerical experiments}\label{section5}
\setcounter{equation}{0}

In this section, we evaluate the numerical performance of the proposed I-IDL-ALM \eqref{I-IDL-ALM} by solving the support vector machine for classification and continuous max-flow models for image segmentation, which can be well modelled in terms of convex minimization problems \eqref{problem} with linear inequality constraints.  The preliminary numerical results illustrate that the new algorithm can converge efficiently with a smaller regularization term and perform competitively with the influential primal-dual algorithm (PDA) proposed in \cite{chambolle2011first} (as a benchmark for comparison). Our algorithms were implemented in Python 3.9 and conducted in a Lenovo  computer with a 2.20GHz Intel Core i7-8750H CPU and 16GB memory.

\subsection{Linear support vector machine (SVM)}

Let us first consider the linear support vector machine (SVM)  model. Suppose the set of training data given by $\mathcal{D}=\big\{(x_i,y_i)\;|\;x_i\in\Re^n,\;y_i\in\{-1,1\},\; i=1,\ldots,m\big\}$ is linearly separable, in which $x$ and $y$ denote the attribute and the label of a sample, respectively.
The linear SVM aims at finding the maximum-margin hyperplane separating two classes of data as much as possible. As discussed in \cite{cristianini2000introduction}, its mathematical form is
\begin{equation*}
 \begin{split}
 \min_{w\in\Re^n,a\in\Re} ~~&\frac{1}{2}{\|w\|}^{2}\\
 \rm{s.t.}~~&y_{i}(w^{T}x_{i}+a)\geq 1, \;\; i=1,\ldots,m.
 \end{split}\end{equation*}
Obviously,  such a model can be compactly regrouped as
\begin{equation}\label{lsvm}
  \min\{\frac{1}{2}{\|Hu\|}^{2}\;|\;Au\geq b\},
\end{equation}
by setting
 $$u=\left(\!\!
       \begin{array}{c}
         w \\
         a \\
       \end{array}
     \!\!\right)
 ,\;\; H=\left(\!\!
                \begin{array}{cc}
                  I_{n\times n} & 0_{n\times 1} \\[0.1cm]
                  0_{1\times n} & 0 \\
                \end{array}
              \!\!\right),\;\;
              A=\left(\!\!
                  \begin{array}{c}
                    y_1({x_1}^T,1) \\[-0.1cm]
                    \vdots \\
                    y_m({x_m}^T,1) \\
                  \end{array}
               \!\! \right)\;\; \hbox{and} \;\;
                b=\left(\!\!
                    \begin{array}{c}
                      1 \\[-0.1cm]
                      \vdots \\
                      1 \\
                    \end{array}
                 \!\! \right).
 $$
When the proposed I-IDL-ALM \eqref{I-IDL-ALM} is applied to \eqref{lsvm}, the resulting scheme is
 \begin{equation}\label{SVMI-ALM}
 \hbox{(I-IDL-ALM)}
   \left\{\begin{array}{ccl}
 \tilde{\lambda}^{k} &=&  [\lambda^k - \beta(Au^k-b)]_{+},\\[0.1cm]
  u^{k+1} &=& (H^TH+\tau r I_{n+1})^{-1}(A^T\tilde{\lambda}^{k}+\tau r u^k), \\[0.1cm]
\lambda^{k+1} &=&  \tilde{\lambda}^k + \beta A(u^k-u^{k+1}).
 \end{array}\right.
 \end{equation}

To simulate, we follow some standard way (e.g., as elucidated on \url{https://web.stanford.edu/~boyd/papers/admm/svm/linear_svm_example.html}) to generate the  random training data satisfying a normal distribution,  and we set $n=2$ for visualizing the classification results. Moreover, the stopping criterion for \eqref{svm} is defined as
$$\|u^{k+1}-u^k\|<10^{-11}.$$
For efficiently implementing the tested algorithms, by selecting out from a great number of different values, we empirically choose the specific parameter settings as follows.
\begin{itemize}
  \item PDA: $r=\sqrt{\rho(A^TA)+0.1}$ and $r=\sqrt{\rho(A^TA)+0.1}$;
  \item  I-IDL-ALM \eqref{I-IDL-ALM}: $\beta=0.01$, $r=\beta(\rho(A^TA)+0.1)$ and $\tau=0.75,1.00,1.20,1.50$.
\end{itemize}

\begin{figure}[H]
\centering
\subfigure[$m=100$]{
\includegraphics[width=8cm]{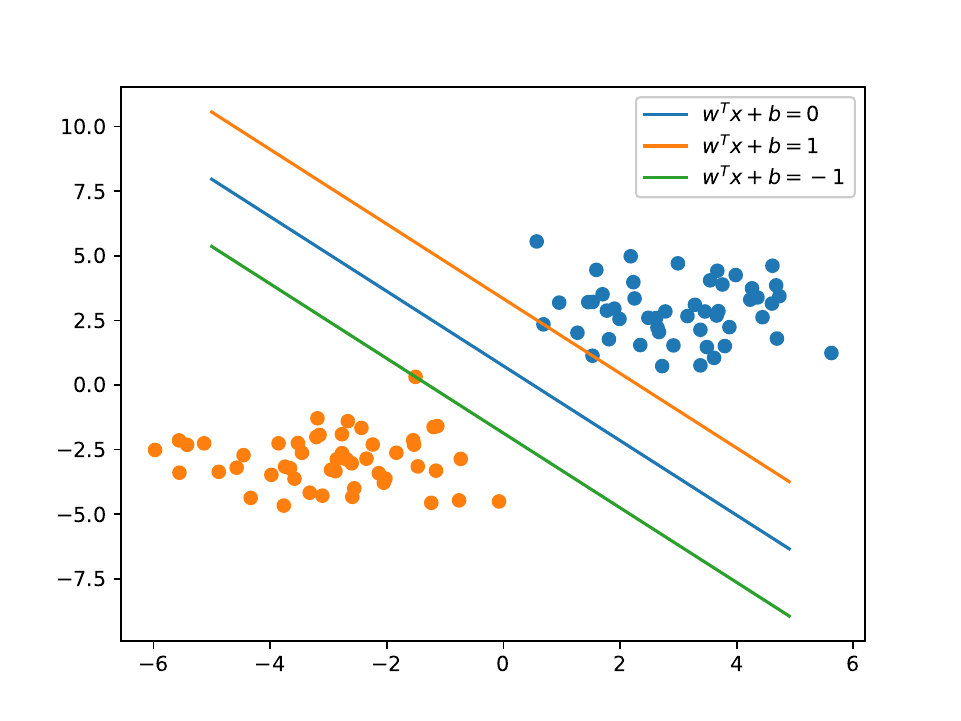}
%\caption{fig1}
}\hspace{-10mm}
\subfigure[$m=100$]{
\includegraphics[width=8cm]{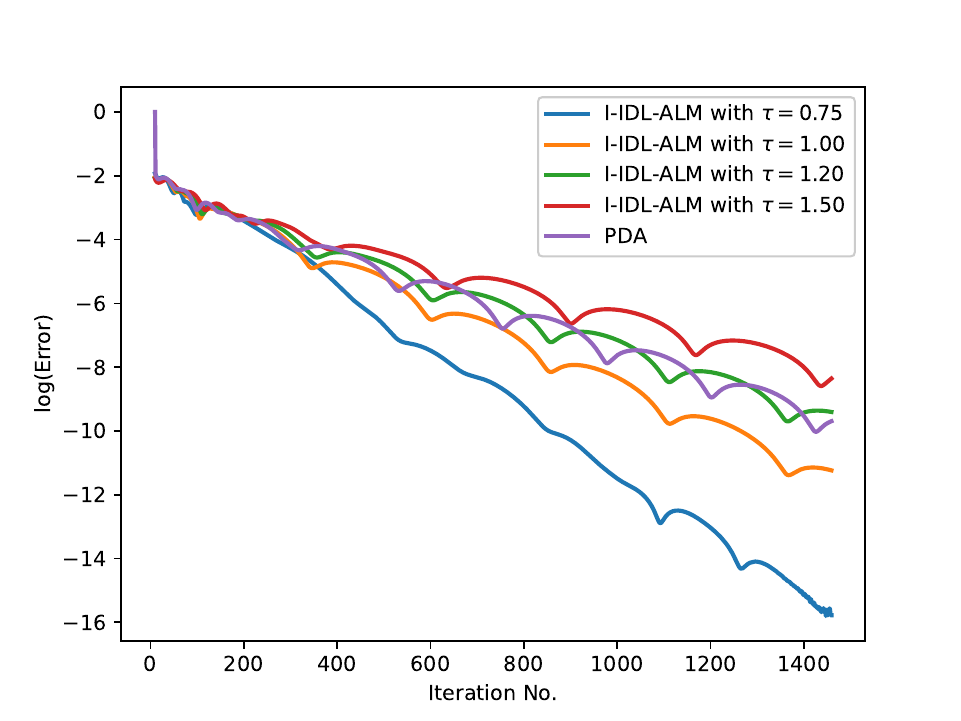}
}
\\ \vspace{-0.25cm}
\subfigure[$m=500$]{
\includegraphics[width=8cm]{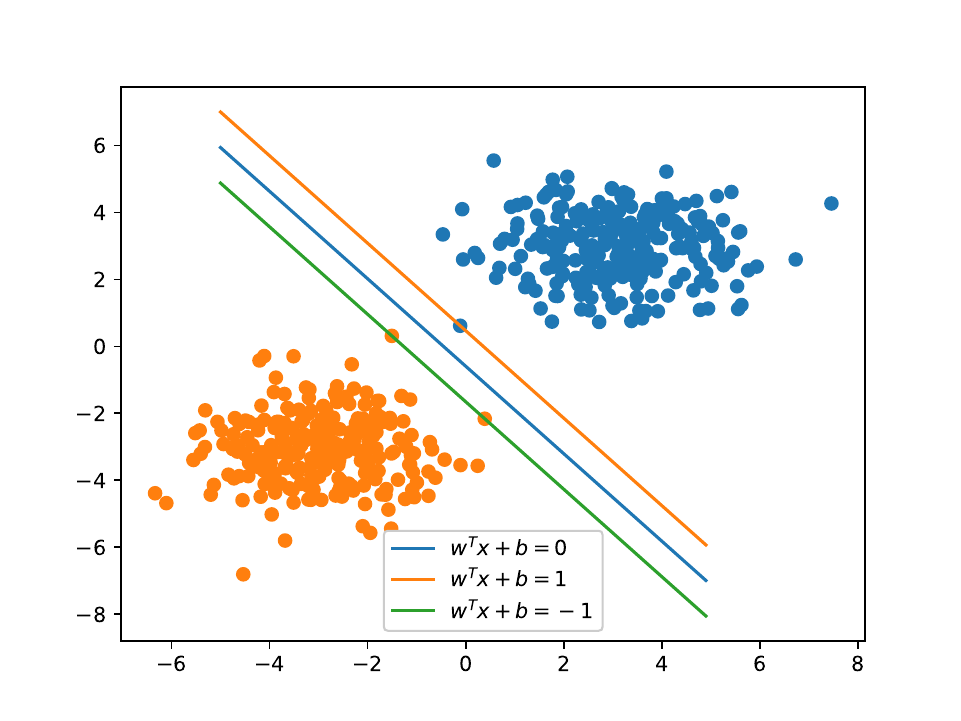}
%\caption{fig1}
}\hspace{-10mm}
\subfigure[$m=500$]{
\includegraphics[width=8cm]{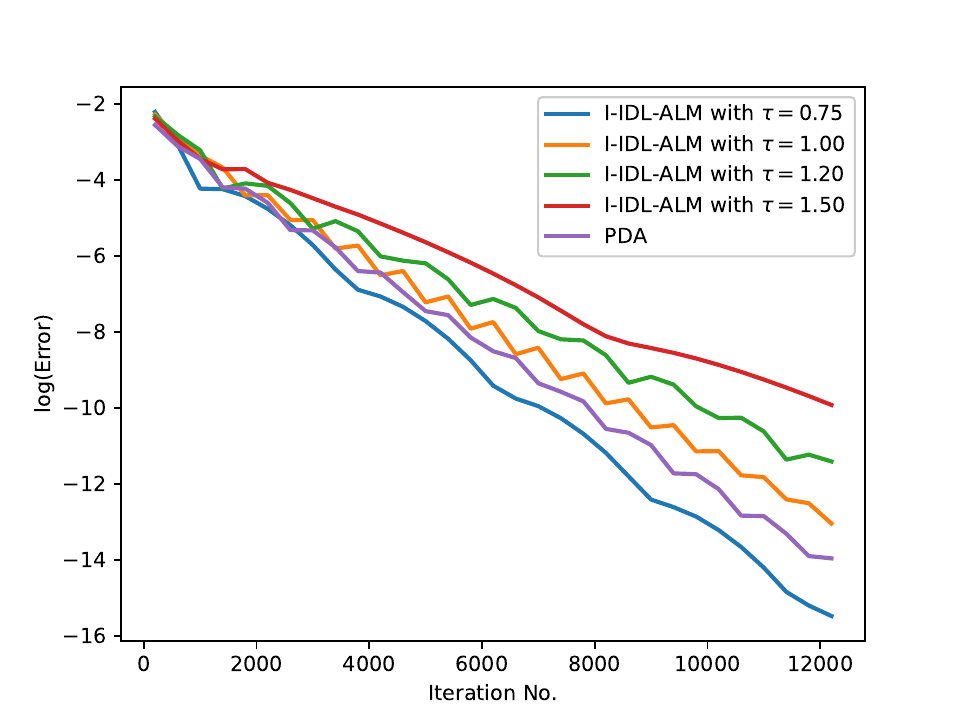}
}\\ \vspace{-0.25cm}
\subfigure[$m=1000$]{
\includegraphics[width=8cm]{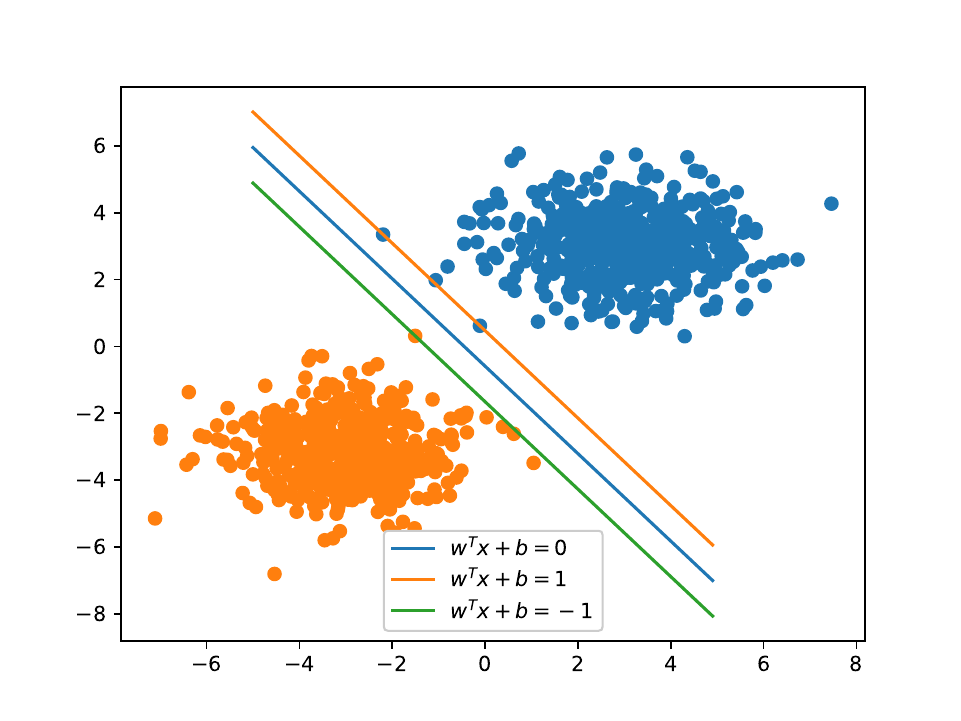}
%\caption{fig1}
}\hspace{-10mm}
\subfigure[$m=1000$]{
\includegraphics[width=8cm]{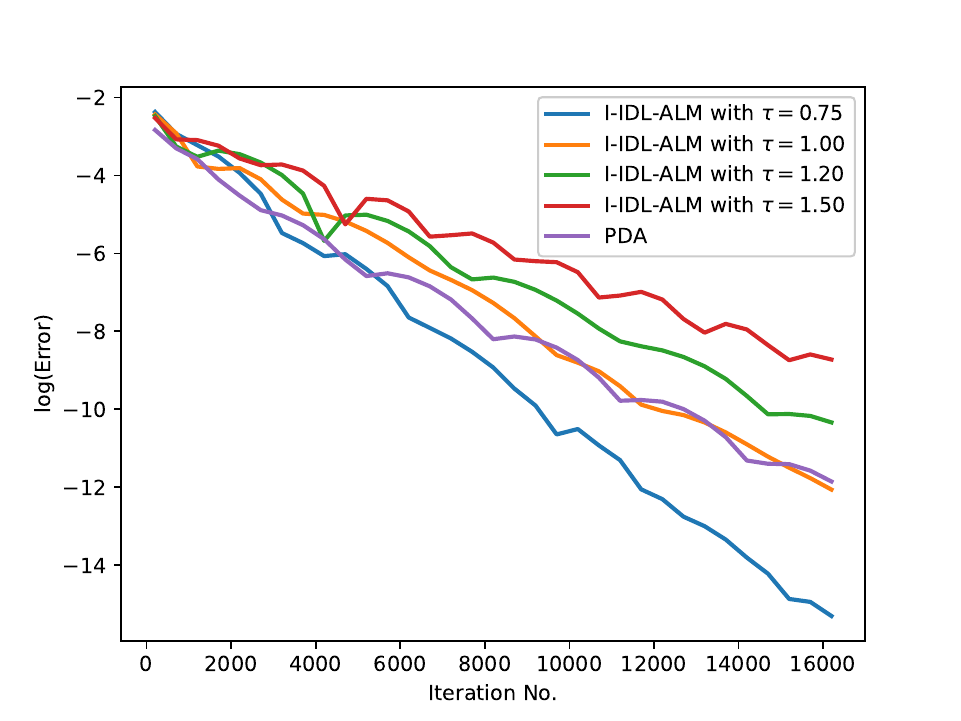}
}\vspace{-0.25cm}
\caption{Left column: some classification results solved by the I-IDL-ALM  \eqref{I-IDL-ALM} with $\tau=0.75$. Right column: the convergence curves of the PDA and the I-IDL-ALM  \eqref{I-IDL-ALM} with various $\tau$ for the linear SVM \eqref{svm} with $m=100,500,1000$. }
\label{svmf}
\end{figure}

\begin{table}[H]
\caption{Numerical results of the linear SVM \eqref{lsvm} with different $m$, solved by the PDA, the I-IDL-ALM \eqref{I-IDL-ALM} with various values of $\tau$.}
\vspace{0.15cm}
\centering
\begin{tabular}{l ccccccc cccccccccc}
\toprule
 \multirow{2}{*}{$m$} &      \multicolumn{2}{c}{PDA}   &   \multicolumn{2}{c}{$\tau=0.75$}  &  \multicolumn{2}{c}{$\tau=1.00$}   &  \multicolumn{2}{c}{$\tau=1.20$} &  \multicolumn{2}{c}{$\tau=1.50$}  \cr \cmidrule(lr){2-3} \cmidrule(lr){4-5} \cmidrule(lr){6-7} \cmidrule(lr){8-9} \cmidrule(lr){10-11}
                                           &   Iter      &   CPU  &   Iter        &   CPU   &   Iter      &   CPU  &   Iter        &   CPU  &   Iter        &   CPU \cr
\midrule
    $100$      &   3635     &    0.96      &   959       &    0.25     &   1346     &    0.34      &   1784     &    0.45      &   2210       &    0.54          \\
    $200$      &   4974     &    1.21      &   2286     &    0.56     &   3133     &    0.77      &   3704     &    0.94      &   4442       &    1.12           \\
    $300$      &   10106   &    2.45      &   6963     &    1.87     &   8283     &    2.24      &   9639     &    2.62      &   11743     &    3.20           \\
    $400$      &   11102   &    3.56      &   8713     &    2.78     &   10410   &    3.34      &   11866   &    3.79      &   14815     &    4.82           \\
    $500$      &   9640     &    3.02      &   7729     &    2.51     &   9491     &    3.07      &   11041   &    3.61      &   13073     &    4.31           \\
    $600$      &   10548   &    3.44      &   8859     &    2.93     &   11024   &    3.64      &   12789   &    4.14      &   15308     &    5.18           \\
    $700$      &   11388   &    3.83      &   9859     &    3.31     &   12432   &    4.10      &   14397   &    4.85      &   17413     &    5.76           \\
    $800$      &   11592   &    3.80      &   9715     &    3.22     &   11196   &    3.70      &   13519   &    4.42      &   17006     &    5.75           \\
    $900$      &   12305   &    4.17      &   10179   &    3.31     &   11602   &    3.77      &   15027   &    5.10      &   19380     &    6.31           \\
    $1000$    &   12861   &    4.18      &   9326     &    3.00     &   13857   &    4.52      &   16984   &    5.54      &   21120     &    7.06           \\
    $2000$    &   32414   &   10.99     &   28580   &    9.64     &   36243   &    12.23    &   41925   &    14.33    &   50770     &    18.28         \\
  \bottomrule
 \end{tabular}
 \label{svm}
\end{table}

In Table \ref{svm}, iteration numbers (``Iter") and computational time (``CPU") are reported for various settings of total data number $m$. As can be seen easily, the new algorithm \eqref{I-IDL-ALM} needs both the less required iteration numbers and computing time for a smaller regularization factor $\tau$.  In particular,  the in-depth choice of $\tau=0.75$ is experimentally shown to be optimal, and  it saves up to $25\%$ iterations compared with the baseline choice of $\tau=1$. We also see from Table \ref{svm} that the new algorithm \eqref{I-IDL-ALM} with $\tau=0.75$ outperforms the influential PDA experimentally, which further shows the efficiency of the proposed method.   To further visualize the numerical results, in Figure \ref{svmf}, we show the computational classification results of the new algorithm \eqref{I-IDL-ALM} with $\tau=0.75$ and plot the convergence curves of the tested algorithms for $100$, $500$ and $1000$ training date sets, respectively. It can be seen again from Figure \ref{svmf} that the new algorithm \eqref{I-IDL-ALM} has a steeper convergence curve for a smaller factor $\tau$.

\subsection{Potts model-based image segmentation}
During recent years, convex optimization has been developed as an efficient and powerful tool to image segmentation \cite{chambolle2011first,chan2006algorithms,lellmann2009convex,yuan2014spatially} with provable  convergence. In this subsection, we consider the convex-relaxed Potts model for multiphase image segmentation
\begin{equation}\label{Potts-ori}
  \min_{u_i(x)\geq0}\;\Big\{\sum_{i=1}^{m}\int_{\Omega}u_i(x)\rho(l_i,x)+\alpha|\nabla u_i(x)|dx  \;\Big | \;\sum_{i=1}^mu_i(x)=1 \Big\},
\end{equation}
where $\alpha>0$ is the weight parameter for the regularization term of the total perimeter of all segmented regions, and $\rho(l_i,x)$ $(i=1,\ldots,m)$ are used to evaluate the cost of assigning the label $l_i$ to the specified position $x$. To avoid directly tackling the complicated pixel-wise simplex constraints on the labeling functions $u_i(x)$, $i=1,\ldots,m$, and the non-smooth total-variation (TV) term (see \cite{ROF1992}) in \eqref{Potts-ori}, we exploit its equivalent dual optimization model introduced in \cite{yuan2010continuous}, namely the continuous max-flow model:
\begin{equation}\label{Potts}
  \begin{aligned}
 \max_{p_s, p, q}\; &\int_{\Omega}p_s(x)dx\\
\;\; \hbox{s.t. } \;& \D {q_i(x)}-p_s(x)+p_i(x)=0,\; i=1,2,\ldots,m;\\
\qquad\; & |q_i(x)|\leq \alpha,\;p_i(x)\leq\rho(l_i,x),\; i=1,2,\ldots,m,
\end{aligned}
\end{equation}
where $p_s(x)$ and $p_i(x)$ $(i=1,\ldots,m)$ denote the source flow and the sink flow respectively, and the linear equality constraints are the so-called flow balance conditions of the continuous max-flow model. The optimized labeling functions $u_i(x)$ $(i=1,\ldots,m)$ by the convex relaxed Potts model are just the multipliers to the linear equality constraints in \eqref{Potts}.
We follow the variational analysis as in \cite{yuan2010study,yuan2010continuous}, where the linear inequality constraints $p_i(x)\leq\rho(l_i,x)$ $(i=1,\ldots,m)$ of \eqref{Potts} are nothing but associated with $u_i(x) \geq 0$, then we can omit both the flow variables $p_i(x)$ and their related linear inequalities so as to fix non-negative labeling functions $u_i(x) \geq 0$ directly. This largely saves computational complexities and therefore results in a new abbreviated version of the continuous max-flow model \eqref{Potts}
as follows:
\begin{equation}\label{Potts-1}
  \begin{aligned}
 \min_{p_s,q} \; &\int_{\Omega} -p_s(x)dx\\
\;\; \hbox{s.t. }\; &\D {q_i(x)}-p_s(x)\geq-\rho(l_i,x),\; i=1,2,\ldots,m;\\
\qquad\; & |q_i(x)|\leq \alpha,\; i=1,2,\ldots,m \, .
\end{aligned}
\end{equation}
It is easy to observe that the above simplified continuous max-flow model \eqref{Potts-1} can be equally written as \eqref{problem} in that
\begin{equation}\label{Potts-L}
  \begin{aligned}
\min_{p_s, q}\; &\int_{\Omega} -p_s(x)dx \, + \, I_C(q(x))\\
 \;\; \hbox{s.t. }\; &\D {q_i(x)}-p_s(x)\geq-\rho(l_i,x),\; i=1,2,\ldots,m,
%\\& \qquad\;  |q_i(x)|\leq \alpha,\; i=1,2,\ldots,m \, .
\end{aligned}
\end{equation}
where $I_C(q(x))$ is the convex characteristic function of the convex set $C := \{|q_i(x)|\leq \alpha,\; i=1,2,\ldots,m\}$.

After discretization, the linearly inequality-constrained convex minimization model \eqref{Potts-L} can be essentially written as
$$
\min_{p_s, q}\; \underbrace{- {\bf 1}^T p_s \, + \, I_C(q)}_{\theta(p_s, q)}
$$
subject to the following linear inequality constraints
$$
\underbrace{\left(\!\!\!
    \begin{array}{c}
      -I \\
      -I \\[-0.1cm]
      \vdots \\
      -I \\
    \end{array}
 \!\!\! \right)p_s(x)+\left(\!\!\!
               \begin{array}{c}
                 \D \\
                 0 \\[-0.1cm]
                 \vdots \\
                 0 \\
               \end{array}
            \!\!\! \right)q_1(x)+\cdots+\left(\!\!\!
                                            \begin{array}{c}
                                              0 \\
                                              0 \\[-0.1cm]
                                              \vdots \\
                                              \D  \\
                                            \end{array}
                                         \! \!\!\right)q_m(x)}_{A (p_s; q_1;\ldots; q_m)}\,
                                         \geq\,
                                         \underbrace{-\left(
                      \!\!\!                                    \begin{array}{c}
                                                            \rho(l_1,x) \\
                                                            \rho(l_2,x) \\[-0.1cm]
                                                            \vdots \\
                                                            \rho(l_m,x) \\
                                                          \end{array}
                   \!\!\!                                    \right)}_{b} \, .
$$
The associated matrix $A$ of \eqref{Potts-L} is thus specified as
$$\left(\!\!
    \begin{array}{ccccc}
      -I & \D &  0 &    \cdots            &      0        \\
      -I        &     0       &   \D &   \cdots        &              0 \\[-0.1cm]
      \vdots &       \vdots  &   \vdots     &        \ddots            &  \vdots    \\
      -I        &     0       &    0     &     \ldots          &             \D \\
    \end{array}\!\!
  \right), \; \hbox{and} \;AA^T=\left(\!\!
                                         \begin{array}{cccc}
                                           I\!-\!\Delta & I            & \cdots & I \\
                                           I            & I\!-\!\Delta & \cdots & I \\[-0.1cm]
                                           \vdots    & \vdots   & \ddots & \vdots \\
                                           I & I & \cdots & I\!-\!\Delta \\
                                         \end{array}\!\!
                                       \right).
$$
Applied the I-IDL-ALM \eqref{I-IDL-ALM} to \eqref{Potts-L}, the resulting scheme is
\begin{equation}\label{algorithmp}
  \left\{
    \begin{array}{cl}
          \tilde{u_i}^k&= \;\big[u_i^k-\beta(-p_s^k+\rho(l_i) + \D q_i^k)\big]_{+},\;\; i=1,\ldots,m,\\[0.2cm]
           p_s^{k+1} & = \; p_s^k + \frac{1}{\tau r}({\bf 1}-\sum_{i=1}^n\tilde{u_i}^k), \\[0.2cm]
           q_i^{k+1}&= \; \mathcal{P}_{C}(q_i^k+\frac{\D^T\tilde{u_i}^k}{\tau r}), \;\;  i=1,\ldots,m,\\[0.2cm]
          {u_i}^{k+1}&=\; \tilde{u_i}^k+\beta \big\{p_s^{k+1}-p_s^k+ \D  (q_i^k-q_i^{k+1})\big\}, \;\;  i=1,\ldots,m.
    \end{array}
  \right.
\end{equation}

To simulate, we follow the package developed by the authors of \cite{yuan2010continuous} (all the codes and image data are available at \url{https://www.mathworks.com/matlabcentral/fileexchange/34224-fast-continuous-max-flow-algorithm-to-2d-3d-multi-region-image-segmentation}), and the stopping criterion for \eqref{Potts-L} is defined as
$$\hbox{ADE}(k):=\frac{\|u^{k+1}-u^{k}\|}{\hbox{size}(u)}<10^{-7},$$
where ``ADE" is short for the average of dual error.  In addition, as analyzed in \cite{Sun2021}, we have $\rho(A^TA)\leq (8+m)$ for the 2D image and $\rho(A^TA)\leq(12+m)$ for the 3D image in the convex relaxed Potts model \eqref{Potts-L}. To implement the tested algorithms efficiently, by selecting out from a number of various values, we choose the concrete parameter settings as follows.
\begin{itemize}
  \item For the case $m=2$, i.e., the foreground-background image segmentation:
  \begin{itemize}
    \item PDA: $r=3$ and $s=9.1/r$ for  2D image, and $r=10$ and $s=13.1/r$ for  3D image;
    \item I-IDL-ALM \eqref{I-IDL-ALM}: $\beta=0.3$, $r=9.1\beta$ and $\tau=0.75,1.00,1.20,1.50$ for 2D image, and $\beta=1.0$, $r=9.1\beta$ and $\tau=0.75,1.00,1.20,1.50$ for 3D image;
  \end{itemize}
  \item For the case $m=4$, i.e., the four-phases image segmentation:
  \begin{itemize}
    \item PDA: $r=2$ and $s=12.1/r$ for  2D image, and $r=2$ and $s=16.1/r$ for  3D image;
    \item I-IDL-ALM \eqref{I-IDL-ALM}: $\beta=0.4$, $r=12.1\beta$ and $\tau=0.75,1.00,1.20,1.50$ for 2D image, and $\beta=0.2$, $r=16.1\beta$ and $\tau=0.75,1.00,1.20,1.50$ for 3D image.
  \end{itemize}
\end{itemize}

\subsubsection{Numerical experiments on foreground-background image segmentation} \label{section-fore-back}

The foreground-background image segmentation aims to partition the given image into two regions of foreground and background, i.e.,  two labels.

In Table \ref{table11}, iteration numbers (``Iter") and computational time (``CPU") are reported for 2D and 3D foreground-background image segmentation. Again, it can be seen easily that the new algorithm \eqref{I-IDL-ALM} converges faster with a small regularization quadratic terms, and that the in-depth choice $\tau=0.75$ is experimentally shown optimal for both the 2D and 3D foreground-background image segmentation. These preliminary numerical results affirmatively verify our theoretical result of the paper, i.e., the relaxation of the regularization term allows a bigger step size to potentially reduce required convergence iterations. In addition, we can discern from Table \ref{table11}  that the new algorithm \eqref{I-IDL-ALM} with $\tau=0.75$  performs competitively with the benchmark PDA, which can be further used to demonstrate the efficiency of the new algorithm. The computed segmentation regions in 2D and 3D images are visualized in Figure \ref{figure1} and Figure \ref{figure2}, respectively.

\begin{table}[H]
\caption{Numerical results of the foreground-background image segmentation with the TV regularization parameters $\alpha=0.5$ (2D image) and $\alpha=0.2$ (3D image), respectively, solved by the PDA, the I-IDL-ALM \eqref{I-IDL-ALM} with various values of $\tau$.}
\vspace{0.15cm}
\centering
\begin{tabular}{l ccccccc cccccccccc}
\toprule
 \multirow{2}{*}{Image} &      \multicolumn{2}{c}{PDA}   &   \multicolumn{2}{c}{$\tau=0.75$}  &  \multicolumn{2}{c}{$\tau=1.00$}   &  \multicolumn{2}{c}{$\tau=1.20$} &  \multicolumn{2}{c}{$\tau=1.50$}  \cr \cmidrule(lr){2-3} \cmidrule(lr){4-5} \cmidrule(lr){6-7} \cmidrule(lr){8-9} \cmidrule(lr){10-11}
                                           &   Iter      &   CPU  &   Iter        &   CPU   &   Iter      &   CPU  &   Iter        &   CPU  &   Iter        &   CPU \cr
\midrule
    2D image    &   212     &    0.70      &   188     &    0.56     &   206     &    0.66      &    213     &    0.67     &    227     &    0.76          \\
    3D image    &   115     &    71.34    &   101     &    75.52   &   117     &    87.51    &    130     &    97.64   &    148     &    110.67       \\
  \bottomrule
 \end{tabular}
 \label{table11}
\end{table}

\begin{figure}[H]
\centering
\subfigure{
\includegraphics[width=5cm]{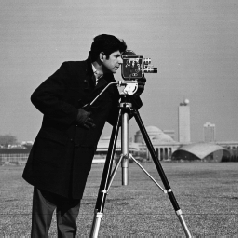}
%\caption{fig1}
}
\qquad\quad
\subfigure{
\includegraphics[width=5cm]{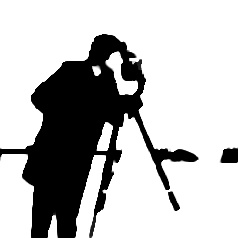}
}
\caption{From left to right: the original image and the computed segmentation result solved by the I-IDL-ALM  \eqref{I-IDL-ALM} with $\tau=0.75$ (image size: $238 \times 238$).}
\label{figure1}
\end{figure}

\begin{figure}[H]
\centering
\subfigure{
\includegraphics[width=4.9cm]{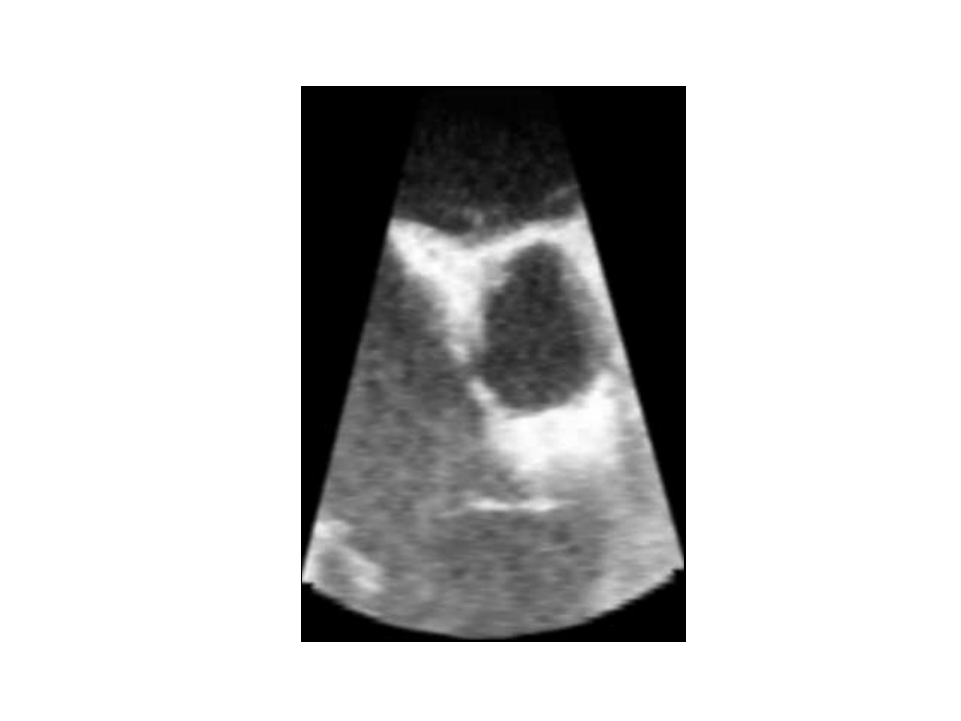}
%\caption{fig1}
}\hspace{-15mm}
\subfigure{
\includegraphics[width=4.9cm]{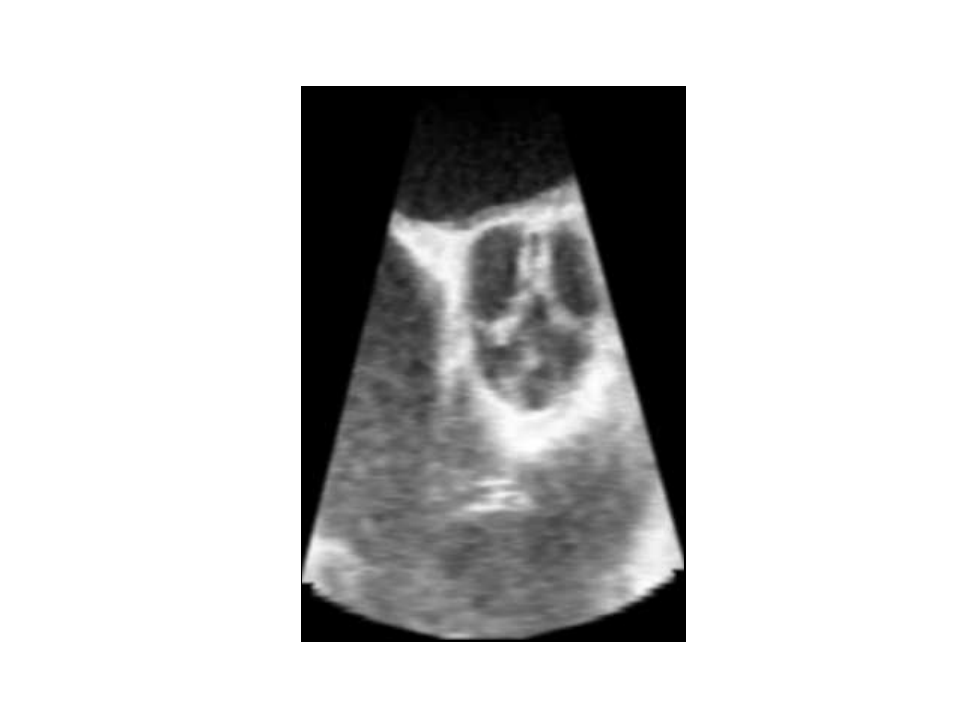}
}\hspace{-15mm}
\subfigure{
\includegraphics[width=4.9cm]{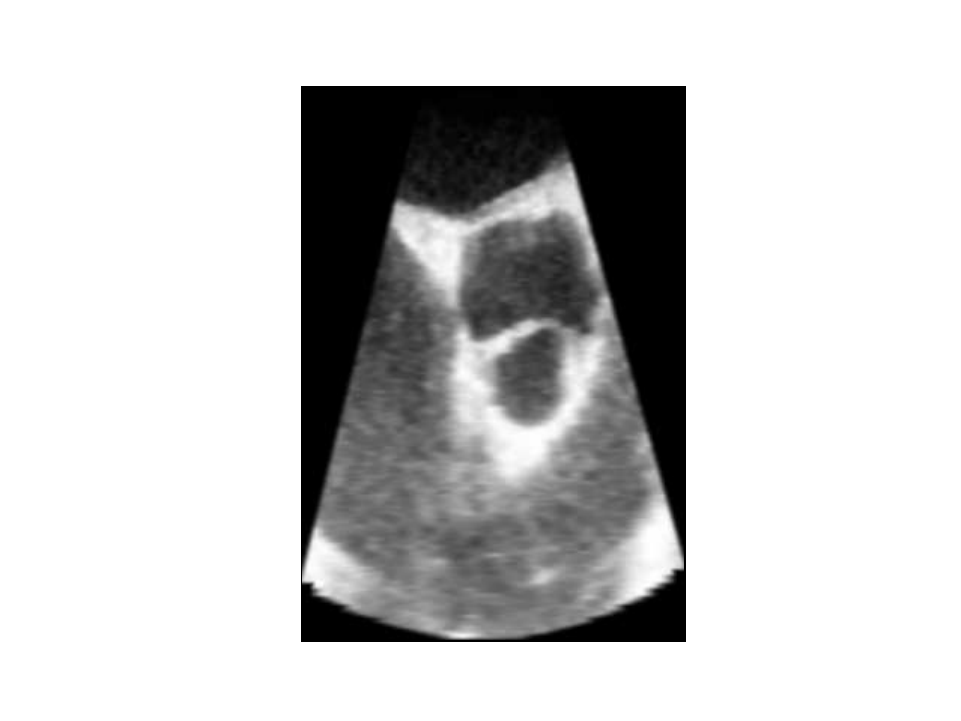}
}\hspace{-15mm}
\subfigure{
\includegraphics[width=4.9cm]{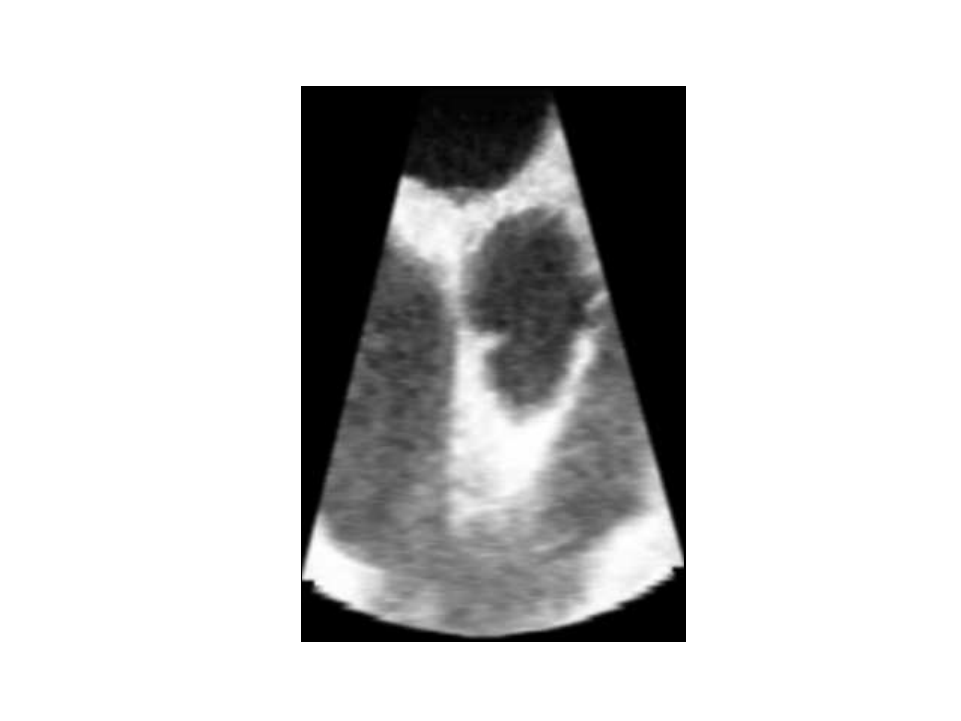}
}\hspace{-15mm}
\\[-0.2cm]
\subfigure{
\includegraphics[width=4.9cm]{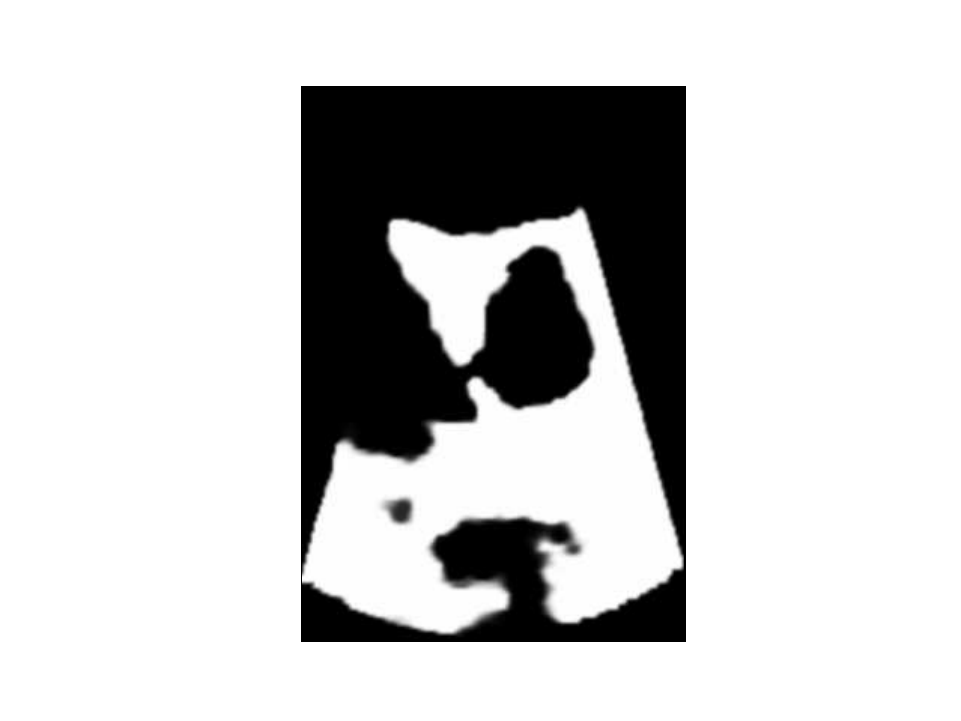}
}\hspace{-15mm}
\subfigure{
\includegraphics[width=4.9cm]{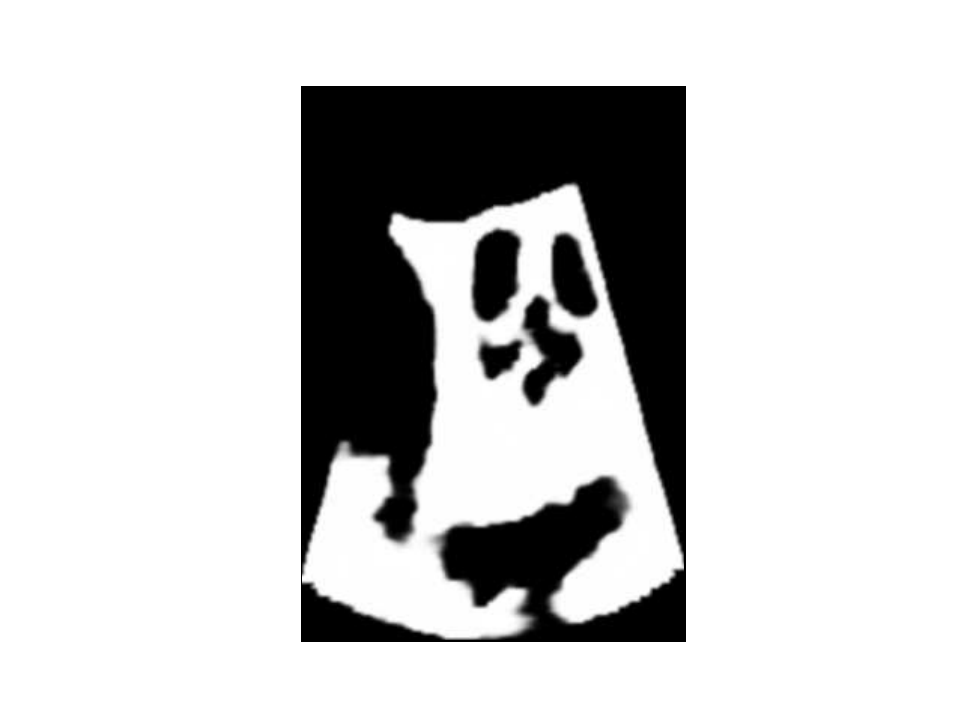}
}\hspace{-15mm}
\subfigure{
\includegraphics[width=4.9cm]{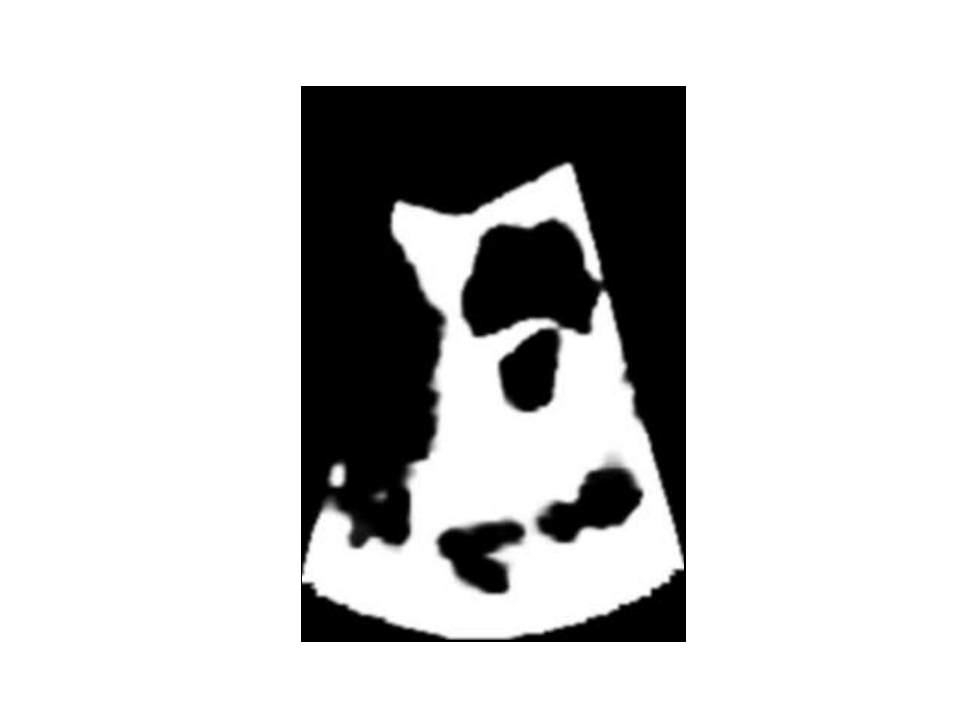}
}\hspace{-15mm}
\subfigure{
\includegraphics[width=4.9cm]{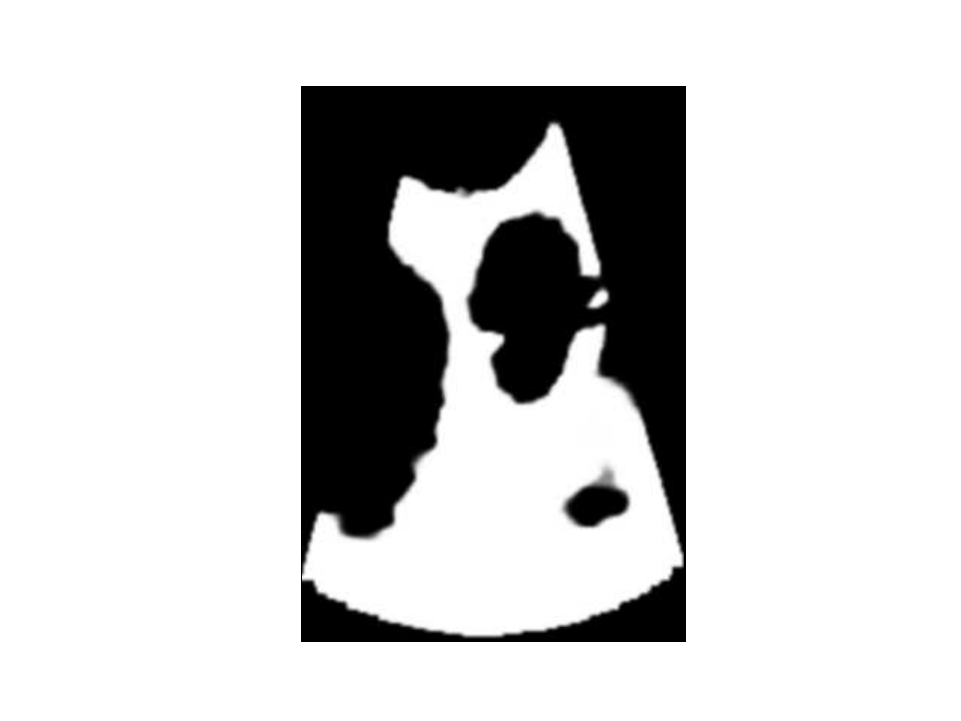}
}\vspace{-0.2cm}
\caption{First row: the original images (the 40-th, 50-th, 60-th, 70-th slice in sagittal view). Second row: the computed segmentation results solved by the I-IDL-ALM  \eqref{I-IDL-ALM} with $\tau=0.75$ (image size: $208\times 112\times 114$).}
\label{figure2}
\end{figure}

\subsubsection{Experiments on multiphase image segmentation}

Numerical experiments on multiphase (4 labels) image segmentation are also conducted in 2D and 3D cases.

Similar to Section \ref{section-fore-back},  in Table \ref{table12}, iteration numbers (``Iter") and computational time (``CPU") are reported for 2D and 3D multiphase (4 labels) image segmentation. For new algorithm \eqref{I-IDL-ALM}, it is easy to see  that a smaller value of $\tau$ results in fewer iterations to render convergence for both 2D and 3D experiments, and the acceleration effect of the I-IDL-ALM \eqref{I-IDL-ALM} with the in-depth choice $\tau=0.75$ is obviously shown again.  It can be seen again from Table \ref{table12}  that the proposed  \eqref{I-IDL-ALM} with $\tau=0.75$ performs competitively with the baseline PDA. Moreover, the computed segmentation regions in 2D and 3D images are visualized in Figure \ref{figure3} and Figure \ref{figure4}, respectively.

\begin{table}[H]
\caption{Numerical results of the multiphase image segmentation (4 labels) with the TV regularization parameter $\alpha=1.0$ (both for 2D and 3D images), solved by the PDA, the I-IDL-ALM \eqref{I-IDL-ALM} with various values of $\tau$.}
\vspace{0.15cm}
\centering
\begin{tabular}{l ccccccc cccccccccc}
\toprule
 \multirow{2}{*}{Image} &      \multicolumn{2}{c}{PDA}   &   \multicolumn{2}{c}{$\tau=0.75$}  &  \multicolumn{2}{c}{$\tau=1.00$}   &  \multicolumn{2}{c}{$\tau=1.20$} &  \multicolumn{2}{c}{$\tau=1.50$}  \cr \cmidrule(lr){2-3} \cmidrule(lr){4-5} \cmidrule(lr){6-7} \cmidrule(lr){8-9} \cmidrule(lr){10-11}
                                           &   Iter      &   CPU  &   Iter        &   CPU   &   Iter      &   CPU  &   Iter        &   CPU  &   Iter        &   CPU \cr
\midrule
    2D image    &   168     &    7.24        &   141      &    5.21       &   166     &    7.90        &   194     &    9.98       &   231       &    11.85          \\
    3D image    &   238     &    315.58    &   219      &    337.90   &   249     &    375.95    &   272     &    417.16   &   306       &    464.51        \\
  \bottomrule
 \end{tabular}
 \label{table12}
\end{table}

 \begin{figure}[H]
\centering
\subfigure{
\includegraphics[width=6.5cm]{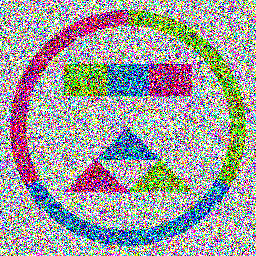}
%\caption{fig1}
}
\qquad
\subfigure{
\includegraphics[width=6.5cm]{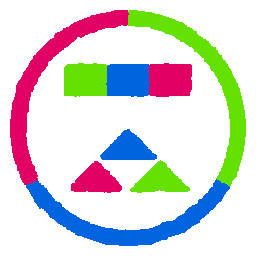}
}
\caption{From left to right: the original image and the computed segmentation result solved by the proposed I-IDL-ALM \eqref{I-IDL-ALM} with $\tau=0.75$ (4 labels, image size: $256\times 256$).}
\label{figure3}
\end{figure}

 In summary,  our observations from these numerical experiments affirmatively demonstrate that the new algorithm \eqref{I-IDL-ALM} can converge faster with a smaller  regularization term. This numerically validates the theoretical results presented in this study.

\begin{figure}[H]
\centering
\subfigure{
\includegraphics[width=4.5cm]{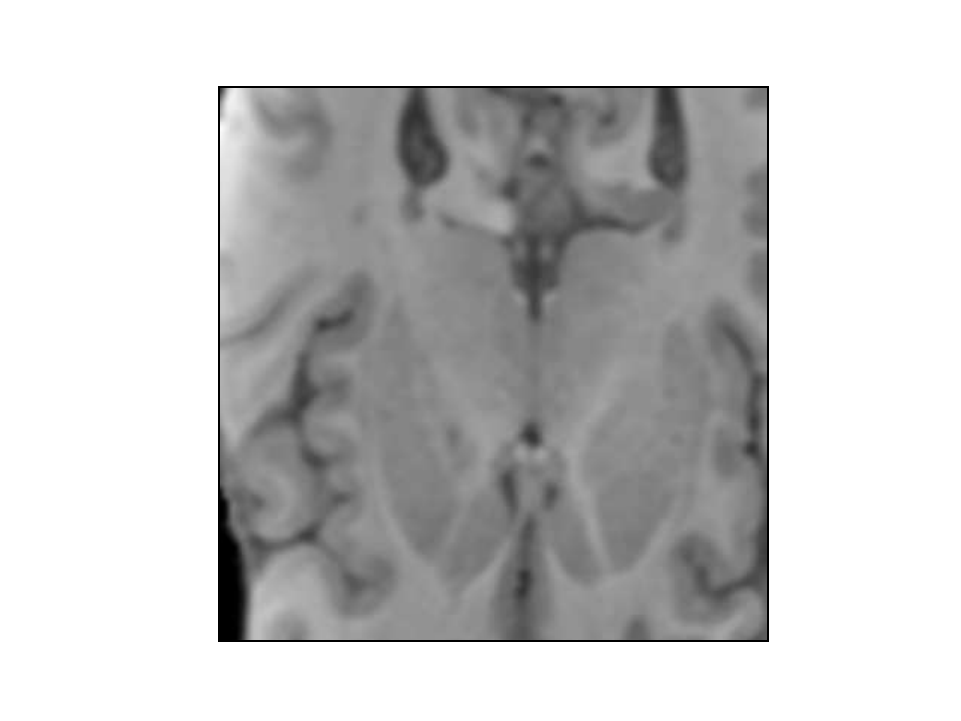}
%\caption{fig1}
}\hspace{-10mm}
\subfigure{
\includegraphics[width=4.5cm]{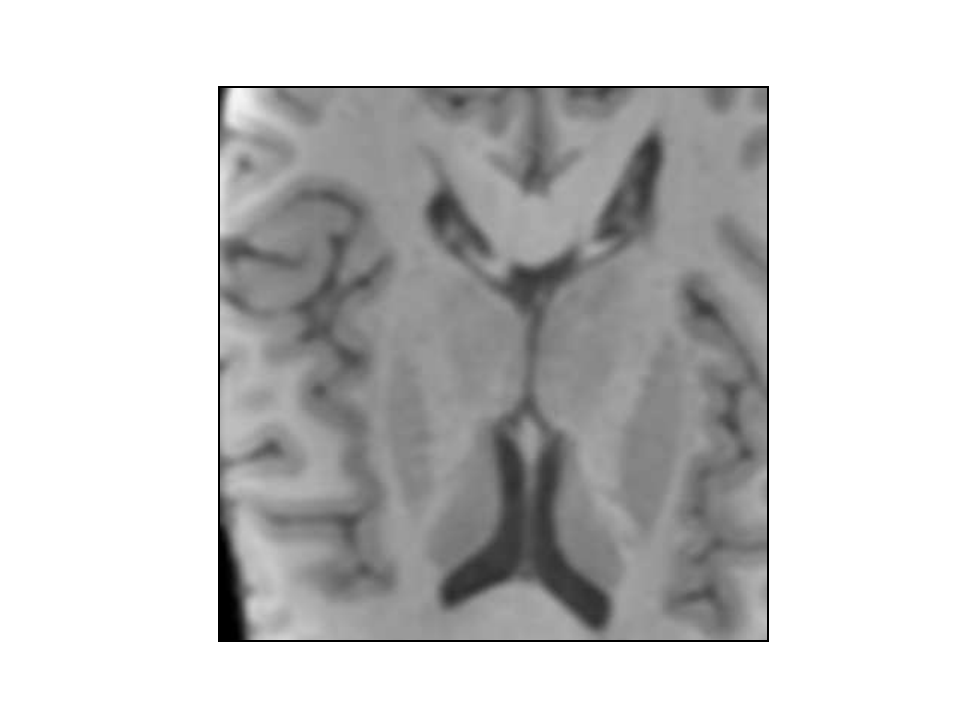}
}\hspace{-10mm}
\subfigure{
\includegraphics[width=4.5cm]{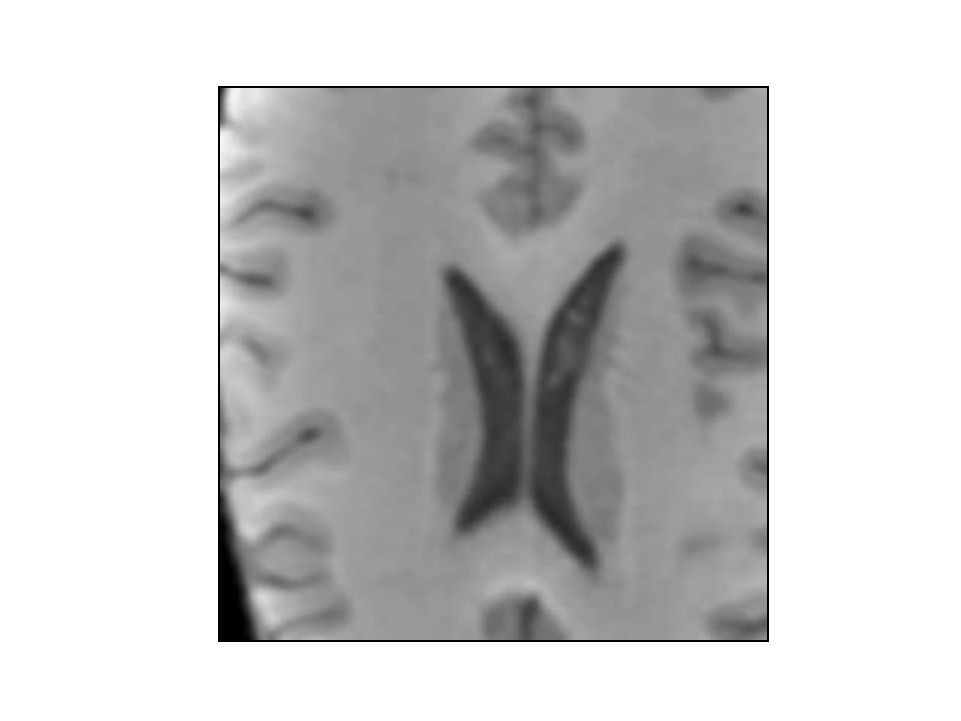}
}\hspace{-10mm}
\subfigure{
\includegraphics[width=4.5cm]{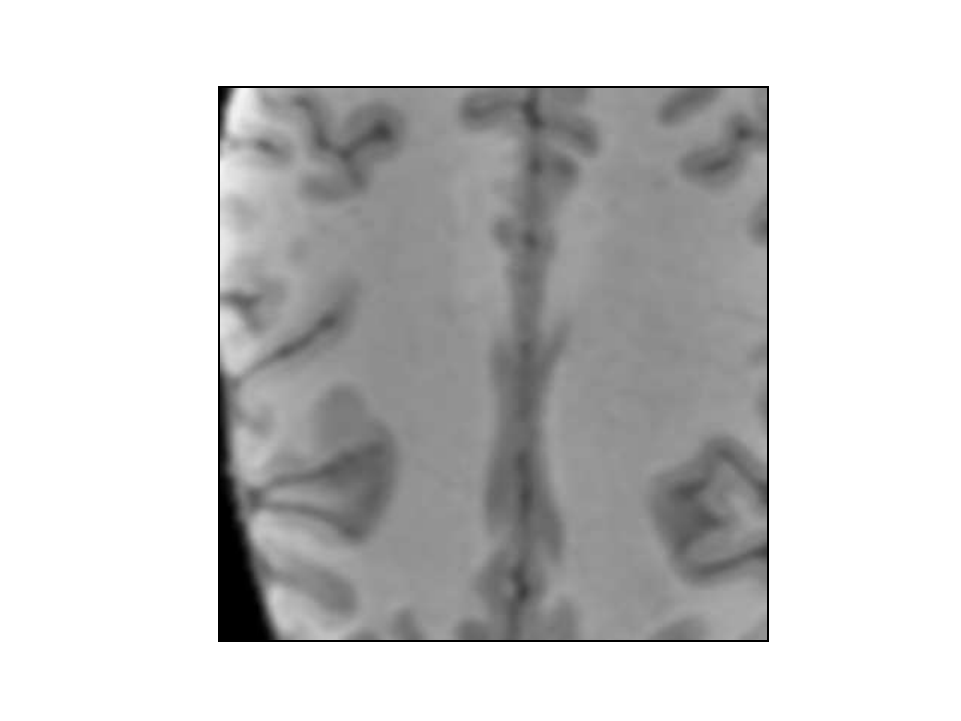}
}%\hspace{-15mm}
\\[-0.35cm]
\subfigure{
\includegraphics[width=4.5cm]{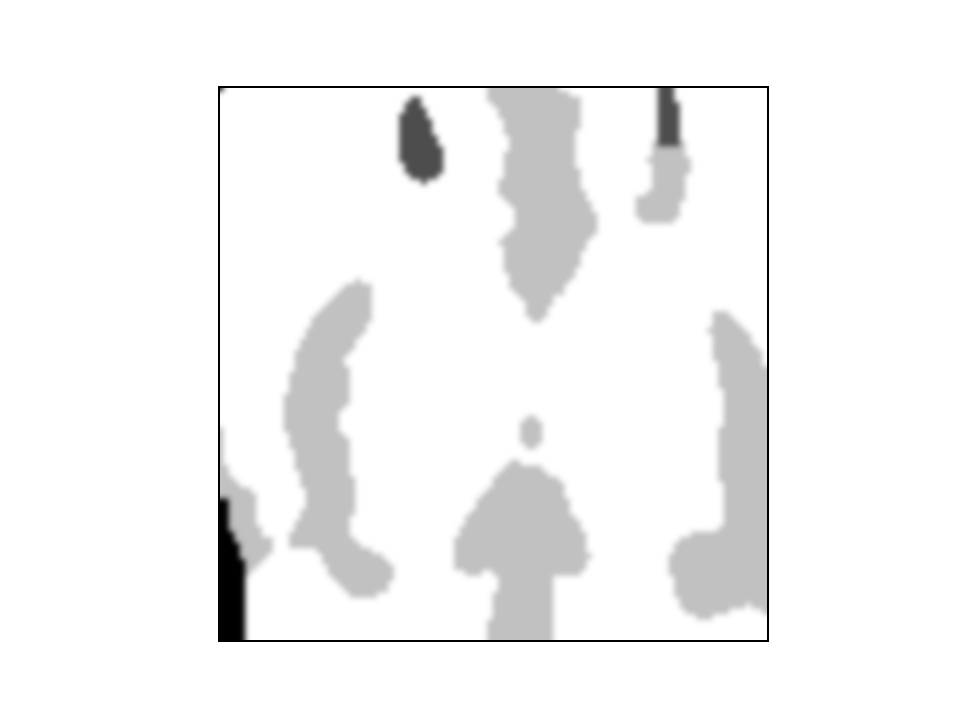}
%\caption{fig1}
}\hspace{-10mm}
\subfigure{
\includegraphics[width=4.5cm]{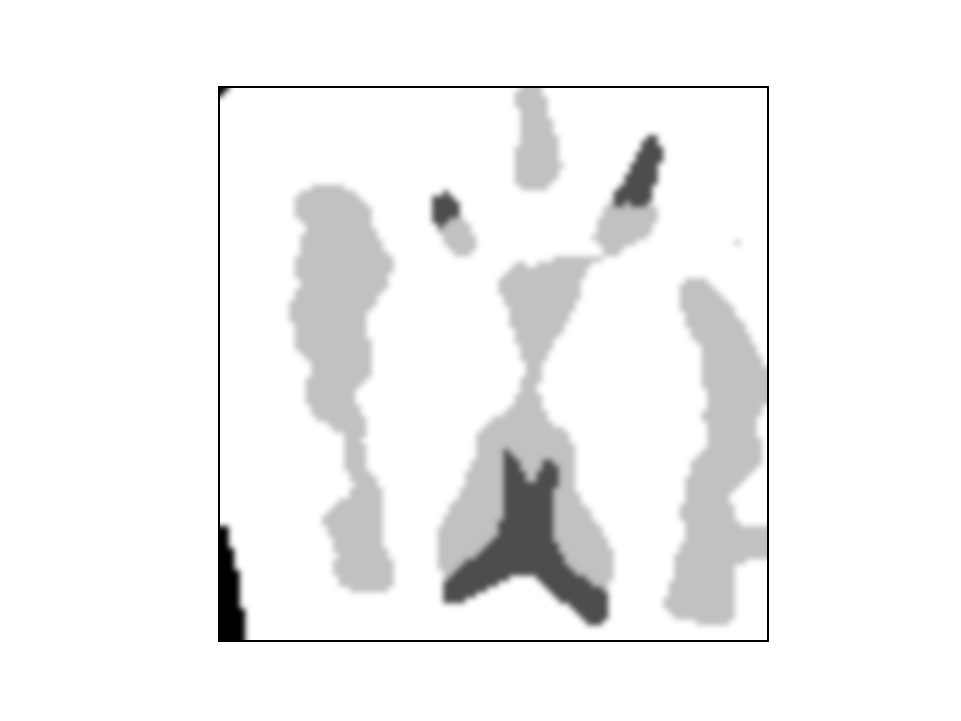}
}\hspace{-10mm}
\subfigure{
\includegraphics[width=4.5cm]{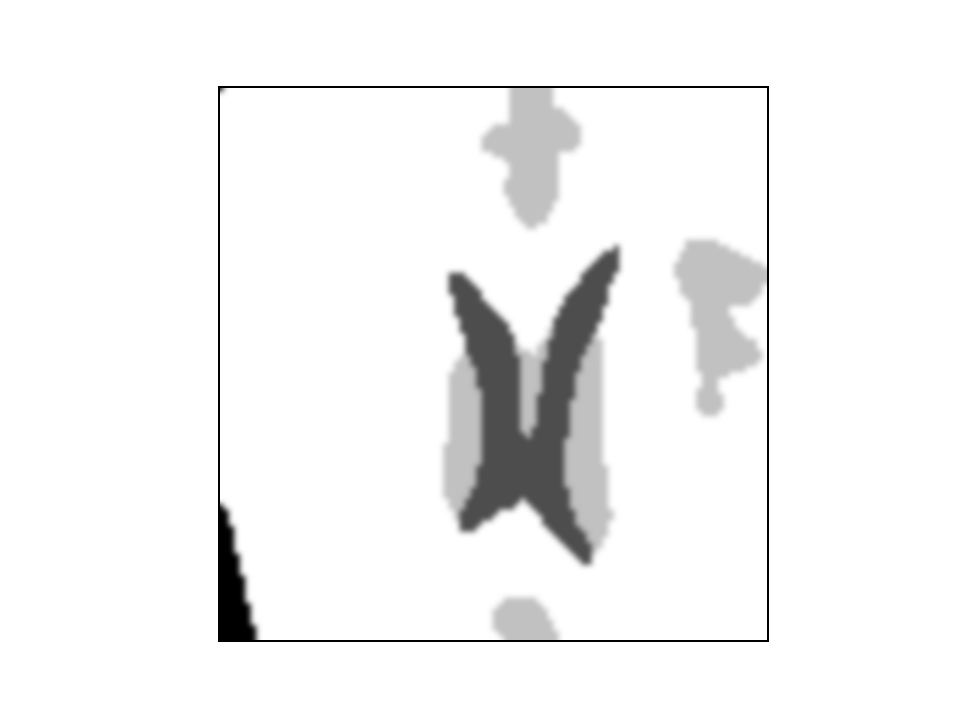}
}\hspace{-10mm}
\subfigure{
\includegraphics[width=4.5cm]{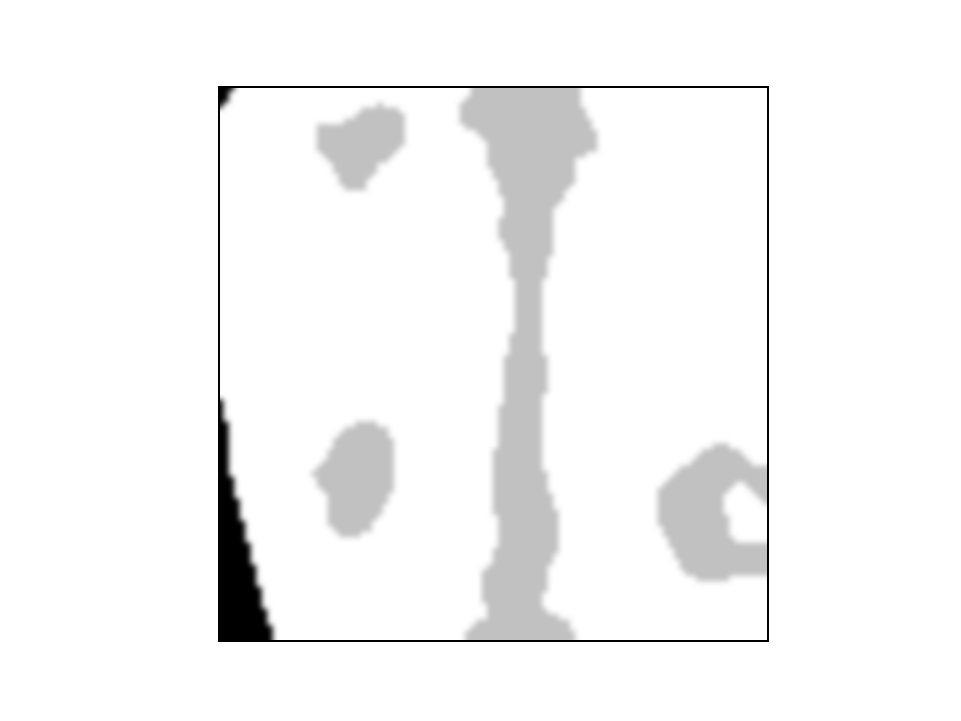}
}%\hspace{-15mm}
\\[-0.35cm]
\subfigure{
\includegraphics[width=4.5cm]{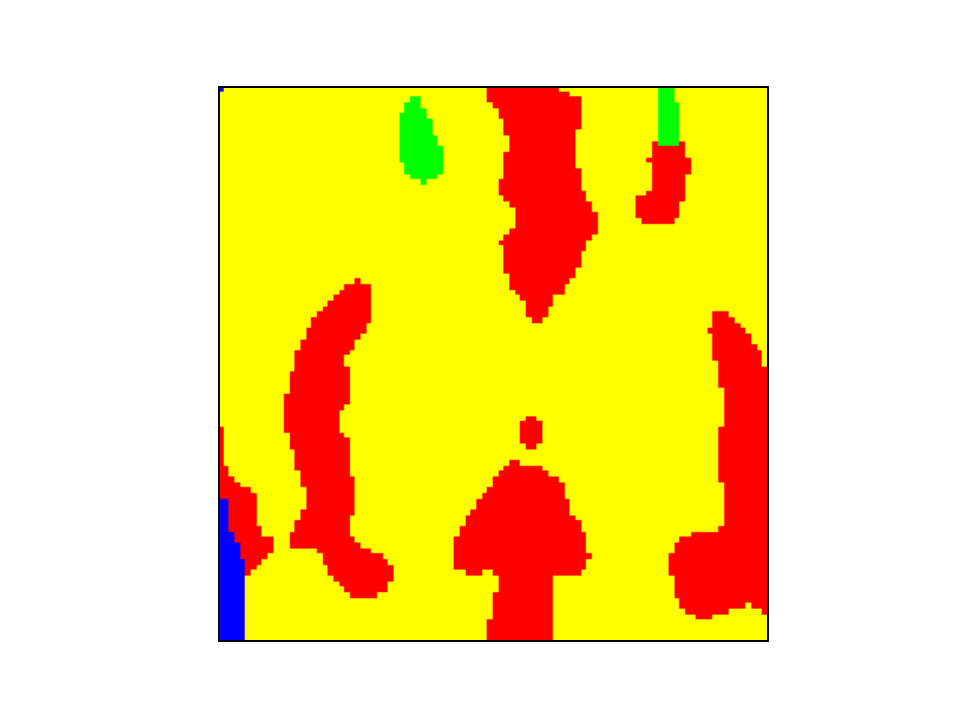}
}\hspace{-10mm}
\subfigure{
\includegraphics[width=4.5cm]{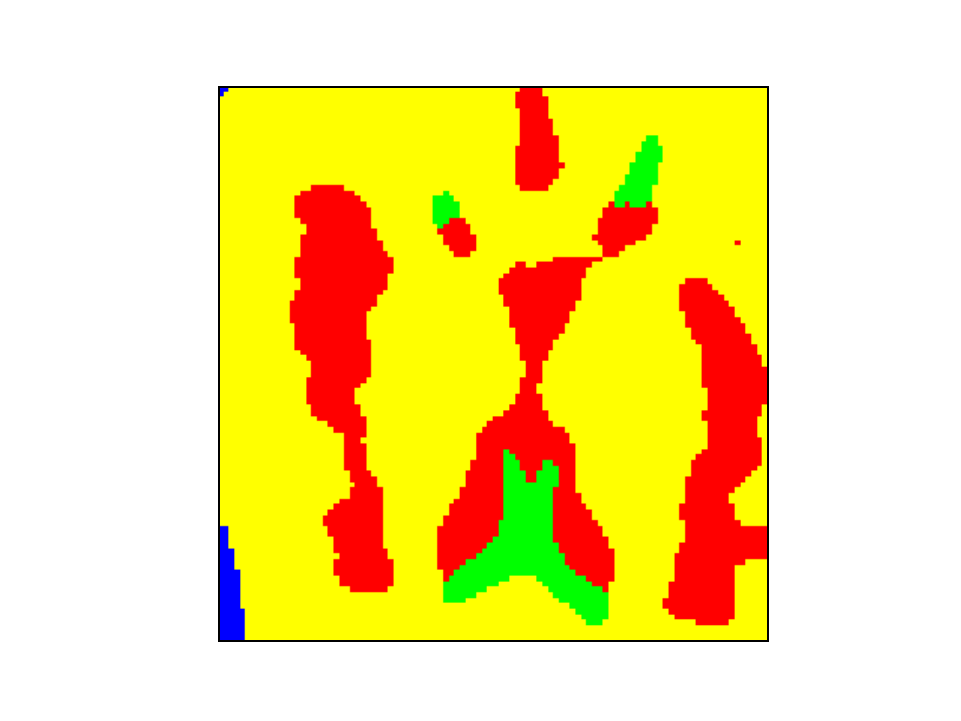}
}\hspace{-10mm}
\subfigure{
\includegraphics[width=4.5cm]{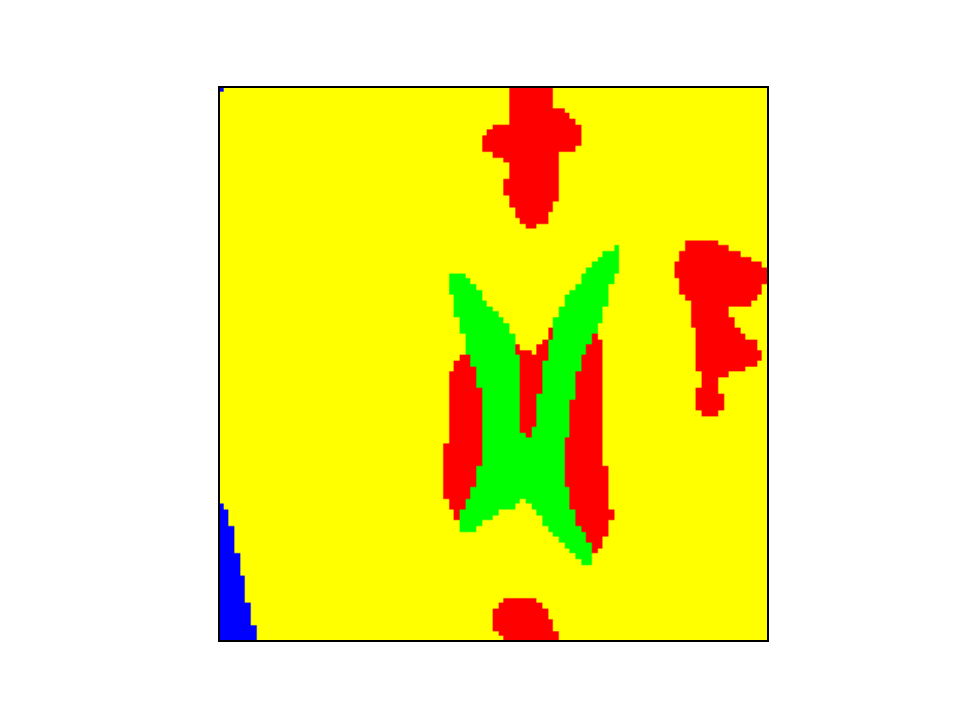}
}\hspace{-10mm}
\subfigure{
\includegraphics[width=4.5cm]{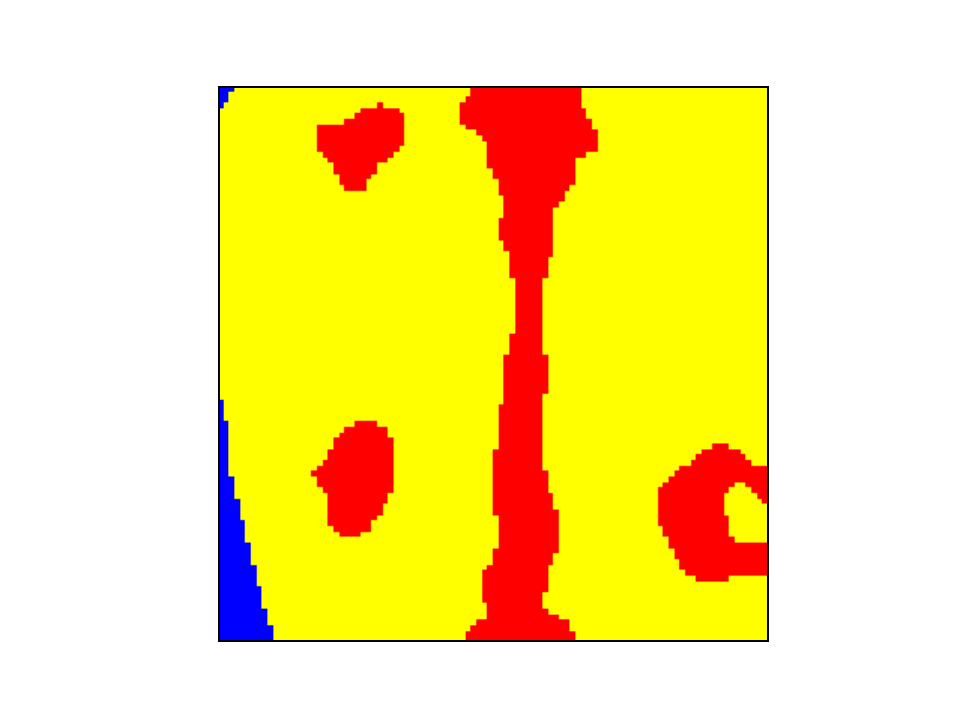}
}
\caption{First row: the original images (70-th, 80-th, 90-th and 100-th slice in transverse view). Second row: the computed segmentation results solved by the  new algorithm \eqref{I-IDL-ALM} with $\tau=0.75$. Third row: the corresponding colorized regions for better visualization (4 labels, image size: $102\times101\times100$).}
\label{figure4}
\end{figure}

%We also evaluate the numerical performance of the proposed method \eqref{I-IDL-ALM} for 4-labels image segmentation problem \eqref{Potts-L}.
%
%In \eqref{table31} and \eqref{table41}, we list the numerical results of the method \eqref{algorithmp} for  4-labels image segmentation problem \eqref{Potts-L}.  \eqref{figure3} and \eqref{figure4} clearly show the associated segmented results.  From these tables,  the acceleration of the I-IDL-ALM \eqref{I-IDL-ALM} with the smallest choice $\tau=0.75$ is obviously shown again for both 2D and 3D image segmentation cases. The convergence curves in \eqref{figure41} demonstrate again that the novel method can converge faster with a smaller  regularization parameter.
%

\section{Extension}\label{section6}
\setcounter{equation}{0}

Note that $\rho(A^TA)$ needs to be estimated beforehand when implementing the proposed method \eqref{I-IDL-ALM} for a concrete model in the form of \eqref{problem}. In this section, we also introduce an indefinite ALM for the studied model \eqref{problem} when $\rho(A^TA)$ is difficult to estimate.

We follow the basic indefinite linearized framework \eqref{I-IDL-ALM} and present the new algorithm as follows:
\begin{subequations}\label{I-Sub-ALM}
\begin{numcases}{ \hbox{\qquad\;\;}}
  \label{Sub-ALM-y}  \tilde{\lambda}^{k} =  [\lambda^k - \beta(Ax^k-b)]_{+},\\[0.15cm]
  \label{Sub-ALM-x}   x^{k+1} = \arg\min\big\{ \theta(x)-(\tilde{\lambda}^k)^TAx+(\tau+\delta)\frac{\beta}{2}\|A(x-x^k)\|^2  \;|\;  x\in {\cal X}   \big\}, \\[0.15cm]
  \label{Sub-ALM-y1} \lambda^{k+1}  =  \tilde{\lambda}^k + \beta A(x^k-x^{k+1}),
  \end{numcases}
\end{subequations}
in which $\beta>0$, $\delta>0$ and $\tau\in(0.75,1)$. It is clear that the proposed I-IDL-ALM \eqref{I-IDL-ALM} is indeed the linearized version of \eqref{I-Sub-ALM}. Also, for the prototypical ALM \eqref{ALM}, by ignoring some terms, it can be regrouped as the following scheme:
\begin{subequations}\label{ALM-Inf}
\begin{numcases}{ \hbox{\qquad\;\;}}
  \label{MALM-y}  \tilde{\lambda}^{k} =  \lambda^k - \beta(Ax^k-b),\\[0.15cm]
  \label{MALM-x}   x^{k+1} = \arg\min\big\{ \theta(x)-(\tilde{\lambda}^k)^TAx+\frac{\beta}{2}\|A(x-x^k)\|^2  \;|\;  x\in {\cal X}   \big\}, \\[0.15cm]
  \label{MALM-y1} \lambda^{k+1}  =  \tilde{\lambda}^k + \beta A(x^k-x^{k+1}).
  \end{numcases}
\end{subequations}
Clearly, the new algorithm \eqref{I-Sub-ALM} has the same computational complexity as the equality ALM \eqref{ALM} (or \eqref{ALM-Inf});  it particularly covers the prototypical ALM \eqref{ALM}  for the case  $\delta+\tau=1$ when the linearly equality-constrained problem \eqref{problem1} is considered; as well as it allows a smaller regularization factor $\delta+\tau\in(0.75,1)$.

%
%Compared with the proposed indefinite linearized ALM method \eqref{I-IDL-ALM}, it is obvious that the scheme \eqref{I-Sub-ALM} takes $(\tau+\delta)\beta\|A(x-x^k)\|^2$ instead of $\tau r\|x-x^k\|^2$ for the subproblem of $x$-approximation,  which avoids solving the difficult subproblem \eqref{IIALM-x} directly and has the same computational complexity as the (equality-constrained) ALM \eqref{ALM}.

To show convergence of the new algorithm \eqref{I-Sub-ALM} for any given $\beta>0$, $\delta>0$ and $\tau\in(0.75,1)$, we first reformulate \eqref{I-Sub-ALM} into a prediction-correction method. Similar to Section \ref{subsec2.3},  by setting $\tilde{x}^k=x^{k+1}$, the predictor of \eqref{I-Sub-ALM} then reads as
 \begin{equation}\label{I-Sub-ALM-P}
 \left\{ \begin{array}{rcl}
 \tilde{\lambda}^{k}& = & [\lambda^k - \beta(Ax^k-b)]_{+},\\[0.25cm]
  \tilde{x}^{k} &=& \arg\min\bigl\{ \theta(x)-(\tilde{\lambda}^k)^TAx+(\tau+\delta)\frac{\beta}{2}\|A(x-x^k)\|^2  \; \big| \;  x\in {\cal X}   \bigr\},
  \end{array} \right.
 \end{equation}
and its corresponding VI-structure satisfies
\begin{center}
 \fbox{\begin{minipage}{16.5cm}
 \medskip
(Prediction step)
$$\tilde{w}^k\in\Omega,\;\;\theta(x)-\theta(\tilde{x}^k)+(w-\tilde{w}^k)^TF(\tilde{w}^k)\geq(w-\tilde{w}^k)^TQ(w^k-\tilde{w}^k),\;\;\forall\;w\in\Omega,$$
where
\begin{equation}\label{Q1}
  Q=\left(\!\!
    \begin{array}{cc}
       (\tau+\delta)\beta A^TA & 0 \\
      -A & \frac{1}{\beta}I_m \\
    \end{array}\!\!
  \right).
\end{equation}
\smallskip
\end{minipage}
}
\end{center}
\medskip
Furthermore, the corrector of \eqref{I-Sub-ALM} can be recursively rewritten as
\medskip
\begin{center}
 \fbox{\begin{minipage}{16.5cm}
 \medskip
(Correction step)

$$w^{k+1}=w^k-M(w^k-\tilde{w}^k),$$
where
\begin{equation}\label{M1}
  M=\left(\!\!
      \begin{array}{cc}
        I_n & 0 \\
        -\beta A & I_m \\
      \end{array}\!\!
    \right).
\end{equation}
\smallskip
\end{minipage}}
\end{center}
\medskip

Under the assumption that $A$ is  full column-rank, $\beta>0$, $\delta>0$ and $\tau\in(0.75,1)$,  we can easily verify that,
$$H=QM^{-1}=\left(\!\!
      \begin{array}{cc}
       (\tau+\delta) \beta A^TA & 0 \\
         0 & \frac{1}{\beta}I_m \\
      \end{array}\!\!
    \right)\succ0,
$$
and
$$ G=Q^T+Q-M^THM=\left(\!\!
      \begin{array}{cc}
         \underbrace{\delta \beta A^TA-(1-\tau)\beta A^TA}_{D_0} & 0 \\[0.2cm]
         0 & \frac{1}{\beta}I_m \\
      \end{array}\!\!
    \right).$$
Setting the new matrices
$$D:=\delta \beta A^TA/\tau \quad \hbox{and}\quad  D_0:=[\delta-(1-\tau)]\beta A^TA,$$
the analysis in Section \ref{sec:4.1} and Section \ref{section4.2} can be repeated seamlessly. Therefore, the global convergence and a worst-case $\mathcal{O}(1/N)$ convergence rate measured
by the iteration complexity can be established for the proposed method \eqref{I-Sub-ALM}.
A formal description in theory will be further exposed in a future work soon.

\section{Conclusions}\label{section7}
%We propose a novel ALM-based algorithm for the canonical convex minimization problem with linear inequality constraints, overcoming the calculation difficulty of the core subproblem in original  inequality-constrained ALM. It generalizes the most recent indefinite linearized ALM from the linearly equality-constrained convex minimization problem to the inequality-constrained issue, so it inherits the advantages of the indefinite linearized (equality-constrained) ALM. The numerical tests on some application problems also demonstrate the efficiency of the proposed method.  This study may enrich and extend our knowledge of the classic inequality-constrained ALM.
%Moreover, the proposed scheme enjoys a parallel structure for solving the core subproblem, which is meaningful for implementing parallel computing when the medical image data are huge.
We present an indefinite linearized ALM for the canonical convex minimization problem with linear inequality constraints, overcoming the calculation difficulty of the core subproblem in the original inequality ALM. To the best of our knowledge, this is the first work to introduce the structure-probed ALM-based scheme for efficiently tackling the canonical convex minimization problem with linear inequality constraints. Under the new algorithmic framework,  the recent indefinite linearized ALM for the linearly equality-constrained convex optimization problem can be regarded as its special case.  The numerical tests on some application problems, including the support vector machine for classification and continuous max-flow models for image segmentation, demonstrate that the proposed method can converge faster with a smaller regularization term.  This study can enrich and extend our knowledge for the inequality version of ALM.

\end{CJK*}

\begin{thebibliography}{11}

\bibitem{bazaraa2006nonlinear}
{\sc M.~S. Bazaraa, H.~D. Sherali, and C.~M. Shetty},  Nonlinear Programming: Theory and Algorithms, John Wiley \& Sons, Hoboken, NJ, 2006.


\bibitem{beck2017first}
{\sc A.~Beck}, First-order Methods in Optimization, SIAM, Philadelphia, 2017.


\bibitem{bertsekas1996constrained}
{\sc D.~P. Bertsekas},  Constrained Optimization and Lagrange Multiplier Methods, Athena Scientific, Belmont, MA, 1996.


\bibitem{Bert2015}
{\sc D.~P. Bertsekas}, Convex Optimization Algorithms, Athena Scientific, Nashua, NH, 2015.


\bibitem{birgin2014practical}
{\sc E.~G. Birgin and J.~M. Mart{\'\i}nez},  Practical Augmented Lagrangian Methods for Constrained Optimization, SIAM, 2014.


\bibitem{boyd2004convex}
{\sc S.~Boyd, S.~P. Boyd, and L.~Vandenberghe},  Convex Optimization, Cambridge University Press, 2004.


\bibitem{chambolle2011first}
{\sc A.~Chambolle and T.~Pock}, A first-order primal-dual algorithm for convex problems with applications to imaging, Journal of Mathematical Imaging and Vision, 40 (2011), pp.~120--145, \url{https://doi.org/10.1007/s10851-010-0251-1}.


\bibitem{chan2006algorithms}
{\sc T.~F. Chan, S.~Esedoglu, and M.~Nikolova}, Algorithms for finding global minimizers of image segmentation and denoising models, SIAM Journal on Applied Mathematics, 66 (2006), pp.~1632--1648,
\url{https://doi.org/10.1137/040615286}.


\bibitem{cortes1995support}
{\sc C.~Cortes and V.~Vapnik},  Support-vector networks, Machine Learning,  20 (1995), pp.~273--297.


\bibitem{cristianini2000introduction}
{\sc N.~Cristianini and J.~Shawe-Taylor}, An Introduction to Support Vector Machines and Other Kernel-based Learning Methods, Cambridge University Press, 2000.


\bibitem{eckstein2017approximate}
{\sc J.~Eckstein and W.~Yao},  Approximate ADMM algorithms derived from Lagrangian splitting, Computational Optimization and Applications, 68 (2017), pp.~363--405, \url{https://doi.org/10.1007/s10589-017-9911-z}.


\bibitem{fortin1983augmented}
{\sc M.~Fortin and R.~Glowinski}, Augmented Lagrangian methods: Applications to the Numerical Solution of Boundary-value Problems, Elsevier, Stud. Math. Appl. 15, North-Holland, Amsterdam, 1983.


\bibitem{glowinski1989augmented}
{\sc R.~Glowinski and P.~Le~Tallec}, Augmented Lagrangian and Operator-splitting Methods in Nonlinear Mechanics, SIAM, Philadelphia, 1989.


\bibitem{he2018}
{\sc B.~He}, My 20 years research on alternating directions method of multipliers, Oper. Res. Trans., 22 (2018), pp.~1--31.


%\bibitem{he2002new}
%{\sc B.~He, L.-Z. Liao, D.~Han, and H.~Yang},  A new inexact alternating directions method for monotone variational inequalities, Mathematical Programming, 92 (2002), pp.~103--118, \url{https://doi.org/10.1007/s101070100280}.


\bibitem{he2014strictly}
{\sc B.~He, H.~Liu, Z.~Wang, and X.~Yuan}, A strictly contractive Peaceman-Rachford splitting method for convex programming, SIAM Journal on Optimization, 24 (2014), pp.~1011--1040, \url{https://doi.org/10.1137/13090849X}.


\bibitem{he2016convergence}
{\sc B.~He, F.~Ma, and X.~Yuan}, Convergence study on the symmetric version of ADMM with larger step sizes, SIAM Journal on Imaging Sciences, 9 (2016), pp.~1467--1501, \url{https://doi.org/10.1137/15M1044448}.


\bibitem{he2020optimal}
{\sc B.~He, F.~Ma, and X.~Yuan}, Optimal proximal augmented Lagrangian method and its application to full Jacobian splitting for multi-block separable convex minimization problems, IMA Journal of Numerical Analysis, 40 (2020), pp.~1188--1216, \url{https://doi.org/10.1093/imanum/dry092}.


\bibitem{he2020optimally}
{\sc B.~He, F.~Ma, and X.~Yuan}, Optimally linearizing the alternating direction method of multipliers for convex programming, Computational Optimization and Applications, 75 (2020), pp.~361--388,  \url{https://doi.org/10.1007/s10589-019-00152-3}.


\bibitem{he20121}
{\sc B.~He and X.~Yuan}, On the {O(1/n)} convergence rate of the Douglas-Rachford alternating direction method, SIAM Journal on Numerical Analysis, 50 (2012), pp.~700--709,
\url{https://doi.org/10.1137/110836936}.


\bibitem{he2018class}
{\sc B.~He and X.~Yuan},  A class of ADMM-based algorithms for three-block separable convex programming, Computational Optimization and Applications, 70 (2018), pp.~791--826, \url{https://doi.org/10.1007/s10589-018-9994-1}.


\bibitem{Hes1969}
{\sc M.~R. Hestenes}, Multiplier and gradient methods, Journal of Optimization Theory and Applications, 4 (1969), pp.~303--320, \url{https://doi.org/10.1007/BF00927673}.


\bibitem{ito2008lagrange}
{\sc K.~Ito and K.~Kunisch}, Lagrange Multiplier Approach to Variational Problems and Applications, SIAM, 2008.


\bibitem{lellmann2009convex}
{\sc J.~Lellmann, J.~Kappes, J.~Yuan, F.~Becker, and C.~Schn{\"o}rr}, Convex multi-class image labeling by simplex-constrained total variation, in International conference on scale space and variational methods in computer vision, Springer, 2009, pp.~150--162.


\bibitem{luenberger1973introduction}
{\sc D.~G. Luenberger}, Introduction to Linear and Nonlinear Programming, vol.~28, Addison-wesley Reading, MA, 1973.


\bibitem{martinet1970breve}
{\sc B.~Martinet},  R{\'e}gularisation d'in{\'e}quations variationnelles par approximations successives, Rev. Fr. Inform. Rech. Oper., 4 (1970), pp.~154--158.


\bibitem{parikh2014proximal}
{\sc N.~Parikh and S.~Boyd},  Proximal algorithms, Foundations and Trends in optimization, 1 (2014), pp.~127--239, \url{https://doi.org/10.1561/2400000003}.


\bibitem{Pow1969}
{\sc M.~J. Powell},  A method for nonlinear constraints in minimization problems, in Optimization, R. Fletcher, ed., Academic Press, New York, 1969, pp.~283--298.


\bibitem{rockafellar1976augmented}
{\sc R.~T. Rockafellar}, Augmented lagrangians and applications of the proximal point algorithm in convex programming, Mathematics of Operations Research, 1 (1976), pp.~97--116, \url{https://doi.org/10.1287/moor.1.2.97}.


\bibitem{rockafellar1976}
{\sc R.~T. Rockafellar}, Monotone operators and the proximal point algorithm, SIAM Journal on Control and Optimization, 14 (1976), pp.~877--898, \url{https://doi.org/10.1137/0314056}.


\bibitem{ROF1992}
{\sc L.~I. Rudin, S.~Osher, and E.~Fatemi}, Nonlinear total variation based noise removal algorithms, Physica D: Nonlinear Phenomena, 60 (1992), pp.~259--268, \url{https://doi.org/10.1016/0167-2789(92)90242-F}.


\bibitem{SNW2011}
{\sc S.~Sra, S.~Nowozin, and S.~J. Wright},  Optimization for Machine Learning, Mit Press, Cambridge, MA, 2011.

\bibitem{Sun2021}
{\sc H.P.~Sun, X.-C.~Tai, J.~Yuan},  Efficient and convergent preconditioned ADMM for the Potts models, SIAM J. Sci. Comput., 43 (2021),  pp.~455--478


\bibitem{yuan2010study}
{\sc J.~Yuan, E.~Bae, and X.-C. Tai}, A study on continuous max-flow and min-cut approaches, in Computer Society Conference on Computer Vision and Pattern Recognition, IEEE, 2010, pp.~2217--2224.


\bibitem{yuan2010continuous}
{\sc J.~Yuan, E.~Bae, X.-C. Tai, and Y.~Boykov}, A continuous max-flow approach to Potts model, in European Conference on Computer Vision, Springer, 2010, pp.~379--392.


\bibitem{yuan2014spatially}
{\sc J.~Yuan, E.~Bae, X.-C. Tai, and Y.~Boykov},  A spatially continuous max-flow and min-cut framework for binary labeling problems, Numerische Mathematik, 126 (2014), pp.~559--587,
\url{https://doi.org/10.1007/s00211-013-0569-x}.


\bibitem{yuan2018modern}
{\sc J.~Yuan and A.~Fenster},  Modern convex optimization to medical image analysis, arXiv preprint, arXiv:1809.08734,  2018.


\bibitem{Zhang2010}
{\sc X.~Zhang, M.~Burger, and S.~Osher},  A unified primal-dual algorithm framework based on Bregman iteration, Journal of Scientific Computing, 46 (2011), pp.~20--46, \url{https://doi.org/10.1007/s10915-010-9408-8}.



\end{thebibliography}
\end{document}